\documentclass{article}
\usepackage{subfig}
\usepackage[section] {placeins}
\usepackage{amsthm}
\usepackage{amssymb,amsmath, mathdots}
\usepackage{anysize}
\usepackage{psfrag}
\usepackage{url}
\newtheorem{theorem}{Theorem}[section]
\newtheorem{definition}{Definition}[section]
\newtheorem{ex}{Example}[section]

\newtheorem{prop}{Proposition}[section]
\newtheorem{cor}{Corollary}[section]

\newtheorem{cl}{Claim}[section]
\newtheorem{lem}{Lemma}[section]
\newtheorem{nota}{Notation}[section]
\numberwithin{table}{section}
\begin{document}

\title{Colorings beyond Fox:\\
the other linear Alexander quandles.}

\author{Louis H. Kauffman\\
Department of Mathematics, Statistics and Computer Science, \\
University of Illinois at Chicago, 851 S. Morgan St., \\
Chicago IL 60607-7045, USA \\
\texttt{kauffman@uic.edu}\\
and\\
Pedro Lopes\\
Center for Mathematical Analysis, Geometry, and Dynamical Systems, \\
Department of Mathematics, \\
Instituto Superior T\'{e}cnico, Universidade de Lisboa\\
1049-001 Lisbon, Portugal \\
\texttt{pelopes@math.tecnico.ulisboa.pt}}

\maketitle

\begin{abstract}
This article is about applications of linear algebra to knot theory. For example, for odd prime p, there is a rule (given in the article) for coloring the arcs of a knot or link diagram from the residues mod p.  This is a knot invariant in the sense that if a diagram of the knot under study admits such a coloring, then so does any other diagram of the same knot. This is called p-colorability. It is also associated to systems of linear homogeneous equations over the residues mod p, by regarding the arcs of the diagram as variables and assigning the equation ``twice the over-arc minus the sum of the under-arcs equals zero'' to each crossing. The knot invariant is here the existence or non-existence of non-trivial solutions of these systems of equations, when working over the integers modulo p (a non-trivial solution is such that not all variables take up the same value).  Another knot invariant is the minimum number of distinct colors (values) these non-trivial solutions require, should they exist. This corresponds to finding a basis, supported by a diagram, in which these solutions take up the least number of distinct values. The actual minimum is hard to calculate, in general. For the first few primes, less than 17, it depends only on the prime, p, and not on the specific knots that admit non-trivial solutions, modulo p.  For primes larger than 13 this is an open problem. In this article, we begin the exploration of other generalizations of these colorings (which also involve systems of linear homogeneous equations mod p) and we give lower bounds for the number of colors.
\end{abstract}



\section{Introduction}
This article is about the application of linear algebra to the theory of knots and links. Knots and links are embeddings of circles into 3-dimensional space \cite{lhKauffman}.  We will use the word ``knot'' to mean both ``knots'' (one component embeddings) and ``links'' (multiple component embeddings); wherever necessary, we will emphasize that we mean a one component link (or a multiple component link). Knot theorists work with knots by projecting them onto a plane, thereby obtaining the so-called knot diagrams. See Figure \ref{fig:diagram-trefoil} for an illustration of a knot diagram. Note that for our purposes, a knot diagram is finite in the sense that it has a finite number of arcs and crossings.
\begin{figure}[!ht]
	\psfrag{a}{\huge$a$}
	\psfrag{b}{\huge$b$}
	\psfrag{c = ma+(1-m)b}{\huge$a\ast b = ma+(1-m)b$}
	\centerline{\scalebox{.35}{\includegraphics{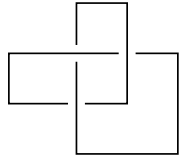}}}
	\caption{A diagram for the trefoil. At the crossings, the broken arc means that in 3-space this arc goes under, locally. }\label{fig:diagram-trefoil}
\end{figure}
\begin{figure}[!ht]
	\psfrag{I}{\huge$I$}
	\psfrag{II}{\huge$II$}
	\psfrag{III}{\huge$III$}
	\centerline{\scalebox{.35}{\includegraphics{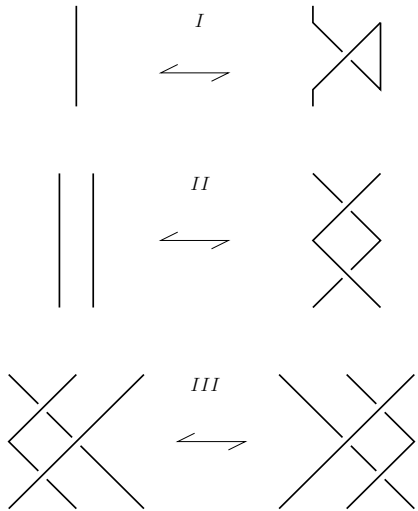}}}
	\caption{The Reidemeister moves, $I, II, III$. These are local transformations; when a diagram undergoes one of these moves, the transformation depicted occurs within a small disc while the rest of the diagram stays the same.}\label{fig:Reidemeister}
\end{figure}
Knots obtained by deformation from a given knot are in the same equivalence class and ultimately thought of as being the same knot. At the level of diagrams, they represent the same equivalence class of knots if and only if they are related by a finite number of Reidemeister moves (this is the Reidemeister Theorem, see \cite{lhKauffman}). Figure \ref{fig:Reidemeister} shows the Reidemeister moves.

One of the main issues in knot theory is to tell apart the equivalence classes referred to above. One way of doing this is by using a knot invariant i.e., a method of associating to a knot another mathematical object such as a number, a polynomial or a group. Furthermore, a knot invariant may be evaluated at  knot diagrams and it remains  the same as one performs Reidemeister moves. In this way, if the  knot invariant evaluated at two  knot diagrams returns  distinct values, the knots they represent are not deformable into one another, they are not in the same equivalence class. Loosely speaking, they are not the same knot. Tri-coloring is an example of such a knot invariant. Fox defined \cite{CFox} a \emph{tri-coloring of a knot diagram} as  a coloring of its arcs with three colors (say red, blue and green) such that (i) more than one color is used and (ii) at each crossing, the three arcs that meet there, all bear the same color or all bear distinct colors. If a diagram satisfies this property it is said \emph{tri-colorable}. Furthermore,
\begin{theorem}
Tri-colorability is invariant under Reidemeister moves i.e., it  is a knot invariant.
\end{theorem}
\begin{proof}The proof is straightforward \cite{lhKauffman} by analyzing what happens to the colors when Reidemeister moves are performed.
\end{proof}
Thus, the trefoil and the unknot are not deformable into one another, as shown in Figure \ref{fig:tri-coloring-trefoil-unknot}. We remark that $p$-colorability for a given odd prime, as defined in the Abstract, is a (first) generalization of tri-colorability. It is simpler to begin with tri-colorability (which is $p$-colorability for $p=3$). However, the minimum number of colors for tri-colorabilty is $3$ (which equals the number of colors available) see Corollary \ref{cor:min3}. Thus we  introduce $p$-colorabilty in the Abstract in order to make sense of the minimum number of colors since it is one of the main invariants dealt with in this article.
\begin{figure}[!ht]
	\psfrag{a}{\huge$a$}
	\psfrag{b}{\huge$b$}
	\psfrag{c = ma+(1-m)b}{\huge$a\ast b = ma+(1-m)b$}
	\centerline{\scalebox{.35}{\includegraphics{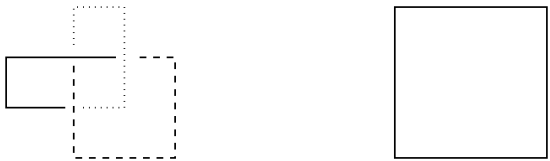}}}
	\caption{(Different sorts of lines stand for different colors.) Left: a diagram of the trefoil, tri-colored. Right: a diagram of the unknot, not tri-colored; since tri-colorability is a knot invariant, the unknot is not tri-colorable, thus not deformable into the trefoil. The trefoil and the unknot are distinct knots.}\label{fig:tri-coloring-trefoil-unknot}
\end{figure}

For a given odd prime $p$, $p$-colorability becomes a problem in linear algebra in the following way. Given a knot diagram, we regard its arcs as algebraic variables and read equations from each crossing  of the following sort: \emph{twice the over-arc minus the sum of the under-arcs equals zero} - the so-called coloring condition - see Figure \ref{fig:coloringcondition}.
\begin{figure}[!ht]
	\psfrag{a}{\huge$a$}
	\psfrag{b}{\huge$b$}
	\psfrag{c}{\huge$c= a\ast b = 2b - a$}
	\centerline{\scalebox{.35}{\includegraphics{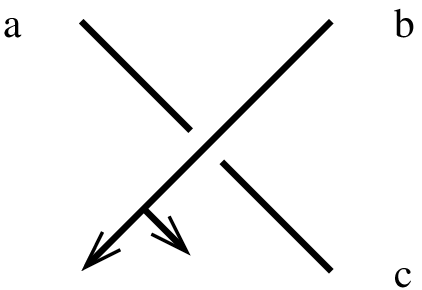}}}
	\caption{Fox coloring condition at a crossing:  $2b-a-c=0$  or \emph{twice the over-arc minus the sum of the under-arcs equals zero}.}\label{fig:coloringcondition}
\end{figure}
We then form a system of linear homogeneous equations over the integers by collecting these coloring conditions over the crossings of the diagram under study. Furthermore,  we consider this system of  equations over the integers modulo $p$, $\mathbf{Z}_p$. $p$-colorability is thus the existence or non-existence of non-trivial colorings i.e., solutions of the system of linear homogeneous equations over $\mathbf{Z}_p$, referred to above, which take up at least two distinct values - for the knot under study. A knot is said to be $p$-colorable  if one of its diagrams admits non-trivial $p$-colorings.

This (linear) algebraic reformulation of $p$-colorability retains a relevant feature of the original one. In fact, any  solution of the homogeneous system of equations mod $p$, is such that, at each crossing, the colors (values) that meet  there are all distinct or are all equal. To see this, just assume the existence of only two colors at a crossing and note that they have to be equal. See the proof of Proposition \ref{prop:min3} below.
\bigbreak

Note that the rule $a\ast b = 2b-a$ is compatible with the relations shown in Figure \ref{fig:Reidemeister+quandle}. That is $a\ast a = a$, $(a\ast b)\ast b = a$, and $(a\ast b)\ast c = (a\ast c)\ast (b\ast c)$, so that Fox colorings are fully compatible with the Reidemeister moves. Note also that the rule $a\ast b = 2b-a$ is a very particular case. An algebraic structure satisfying these axioms is called a \emph{quandle} - see Section \ref{sect:quandles}.

\bigbreak
We resume considering the rule $a\ast b = 2b-a$ where $a$ is one under-arc and $b$ is the over-arc of the crossing at issue. We remark that upon performance of a Reidemeister move on the diagram, as in Figure \ref{fig:Reidemeister+quandle}, there is a unique assignment of variables to arcs and appearance or disappearance of equations associated to the new diagram, induced by those in the original diagram \cite{pLopes}. The new system of coloring conditions is such that its matrix of coefficients relates to the original one by elementary moves on matrices. Thus, the equivalence class of these systems of linear homogeneous equations over the integers modulo elementary moves on matrices, constitutes a knot invariant. In particular, the number of solutions of these systems of equations is a knot invariant. As for the solutions, there are always the trivial ones, the ones that assign the same value (henceforth color) to each arc in the diagram. The existence of this sort of solution is related to the fact that the determinant of the coefficient matrix (henceforth coloring matrix) is zero: each row is formed by exactly one $2$, two $-1$'s and $0$'s; thus, adding all columns we obtain a column of $0$'s.

 \bigbreak
For knots, it turns out that we can take any $(n-1)\times (n-1)$ minor matrix of the $n\times n$ coloring matrix and that minor matrix will have non-zero and odd determinant (for one-component knots) \cite{CFox}. The system of equations corresponding to this minor matrix can be interpreted as the coloring system of equations where we have set equal to zero one selected arc of the diagram. Since the determinant of the minor matrix is not zero, when it further is not equal to $\pm 1$ we can produce non-trivial solutions of the system of equations mod d where d is the absolute value of the aforementioned minor determinant. These results are independent of the choice of the minor matrix. Thus, $d$ is another knot invariant, called \emph{the determinant of the knot} under study. This is the beginning of the topic of  {\it Fox colorings} for knots. Going back to the solutions of our system of equations, we can obtain non-trivial solutions  by considering the system of equations over the integers mod $f$, $\mathbf{Z}_f$, where $f$ is any factor of $d$ greater than $1$. This is how we choose the $p$ in $p$-colorability. It is any of the prime factors of the determinant of the knot under study.

Here is another perspective on the coloring matrix. Since it is its equivalence class that matters (it is the knot invariant) we might as well work with a preferred representative of the equivalence class. Let us choose the Smith Normal Form of the coloring matrix. It has  at least one $0$ along the diagonal, because the determinant of the coloring matrix is zero. The elements of the diagonal, modulo sign, constitute also a knot invariant. (These elements can be interpreted in terms of the homology of a 2-fold branched covering along the knot of $S^3$ \cite{Przytycki}.) The absolute value of their product, except for the $0$ referred to above, is called the \emph{determinant of the knot} under study (equivalent to the definition in the previous paragraph). It is a fact that the  determinant for knots (one component links) is an odd integer.  Going back to the solutions of our system of equations, we can obtain non-zero solutions to the system by working in an appropriate modular number system. We do that by choosing a prime factor of one of the non-zero elements, say $p$, of this diagonal and working  over the integers modulo $p$, $\mathbf{Z}/p\mathbf{Z}$. In particular, if $p=3$, then the knot is tricolorable. We remark that our choice of a prime factor $p$ has to do with the fact that in this way the associated ring of modular integers is a field and we can use the techniques of standard linear algebra. We can also work over $\mathbf{Z}_m$ for composite $m$, if we wish, by doing linear algebra over commutative rings \cite{McDonald}. We can now say that if the determinant of the knot under study is divisible by an odd prime $p$, then the knot is \emph{$p$-colorable} or \emph{colorable mod $p$}.

\bigbreak

\subsection{Results and organization of this article.}\label{subsect:results+org}

Fox coloring is based on the rule $a\ast b = 2b-a$. In fact, if $T$  is an algebraic variable, then $a\ast b = Ta + (1-T)b$ satisfies the quandle rules, as illustrated in Figure \ref{fig:Reidemeister+quandle}. This means that we can generalize Fox colorings by working over $\mathbf{Z}_p$ (for a given odd prime $p$) and using a non-null integer $m$, with $a\ast b =ma+(1-m)b$, mod $p$. We call such quandles \emph{Linear Alexander Quandles} \cite{Nelson}. If a knot diagram can be assigned integers to its arcs such that at each crossing the condition $a\ast b =ma+(1-m)b$, mod $p$ is satisfied and at least two arcs are assigned distinct colors (i.e., integers mod $p$), then the corresponding link is said to admit non-trivial $(p, m)$-colorings.
\begin{definition}\label{def:mincolpm}
Let $K$ be a knot admitting non-trivial $(p, m)$-colorings. Let $D$ be a diagram of $K$ and let $n_{D, p, m}$ be the minimum number of colors it takes to equip $D$ with a non-trivial $(p, m)$-coloring. We let $$mincol_{p, m} (K) = \min \{ n_{D, p, m} \, |\, D \text{ is a diagram of } K \}$$ and refer to it as {\bf the minimum number of colors for non-trivial $(p, m)$-colorings of $K$}. When the context is clear we will say {\bf the minimum number of colors for $K$}.

\end{definition}The rest of the article is devoted to these matters. In particular, the main Theorem of this article is (Section \ref{sect:palette})
\begin{theorem}\label{thm:satohmin}
Let $K$ be a knot  i.e., a $1$-component link. Let $p$ be an odd prime. Let $m$ be an integer such that $K$ admits non-trivial $(p, m)$-colorings. If $m \neq 2$ $($or $m=2$ but $\Delta^0_K(m)\neq 0$, explained below$)$ then $$2 + \lfloor \ln_M p \rfloor \leq mincol_{p, m} (K) ,$$ where $M = \max \{ |m|, |m-1|  \}$.
\end{theorem}

In the set up of Fox colorings the analogous theorem has already been proven  in \cite{NNS}. Our proof of Theorem \ref{thm:satohmin} is a generalization of their work. This is done in Section \ref{sect:palette}.

The rest of the article is organized as follows. Section \ref{sect:quandles} introduces the algebraic structure which constitutes the underpinning of our considerations - the quandle; examples are given. Homomorphisms of quandles have as particular examples Fox colorings (which take as target quandles the so-called dihedral quandles) - this is Subsection \ref{subsect:Alex-Fox-other-colorings}. Instead, in this article, we use the other linear Alexander quandles - which are introduced in Subsection \ref{subsect:metho}.

In Section \ref{sect:examples} we present examples of non-trivial $(p, m)$-colorings. We present results on automorphisms of colorings (Section \ref{sect:autocolorings}), on obstructions to colorings (Section \ref{sect:obstructions}) and minimizations by direct calculations (Section \ref{sect:minimizing}).  In Section \ref{sect:reducing} we reduce the number of colors in two examples. Finally, in Section \ref{sect:futurework} we collect a few questions for future work.

\subsection{Acknowledgements.}\label{subsec:ack}

L. K. is pleased to thank the Simons Foundation (grant number 426075) for partial support for this research under his Collaboration Grant for Mathematicians (2016 - 2021).
P.L. acknowledges partial support from FCT (Funda\c c\~ao para a Ci\^encia e a Tecnologia), Portugal, through project FCT EXCL/MAT-GEO/0222/2012, ``Geometry and Mathematical Physics''.

\section{Quandles}\label{sect:quandles}

The quandles \cite{DJoyce, SMatveev} were conceived in order to make an algebraic version of the Reidemeister moves. We give the definition of quandles below in Definition \ref{def:quandle} and illustrate the connection between Reidemeister moves and quandle axioms in Figure \ref{fig:Reidemeister+quandle}.
\begin{figure}[!ht]
	\psfrag{I}{\huge$I$}
	\psfrag{II}{\huge$II$}
	\psfrag{III}{\huge$III$}
	\psfrag{x}{\huge$x$}
	\psfrag{a}{\huge$a$}
	\psfrag{b}{\huge$b$}
	\psfrag{c}{\huge$c$}
	\psfrag{aast a=a}{\huge$a\ast a = a$}
	\psfrag{aast b}{\huge$a\ast b$}
	\psfrag{bast c}{\huge$b\ast c$}
	\psfrag{xastb = aast b}{\huge$x = (a\ast b)\bar{\ast} b \Longleftrightarrow x = a$}
	\psfrag{a astb ast c}{\huge$(a\ast b)\ast c$}
	\psfrag{aast cast bast c}{\huge$(a\ast c)\ast (b\ast c)$}
	\psfrag{abc}{\huge$(a\ast b)\ast c = (a\ast c)\ast (b\ast c)$}
	\centerline{\scalebox{.35}{\includegraphics{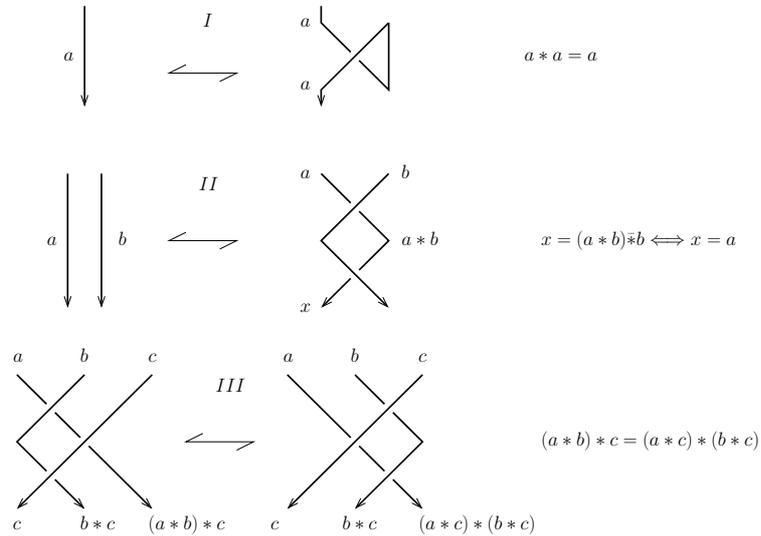}}}
	\caption{The Colored Reidemeister moves, $I, II, III$ where the arcs are generators and relations are read at crossings as indicated; equating left-hand  with right-hand sides for each move gives rise to the quandle axioms - see Definition \ref{def:quandle} below.}\label{fig:Reidemeister+quandle}
\end{figure}
\begin{definition}[Quandle]\label{def:quandle} A \emph{quandle} is a set, $X$, equipped with a binary operation, $\ast$, such that
\begin{enumerate}
\item For any $a\in X$, $a\ast a = a$;

\item For any $a, b\in X$, there is a unique $x\in X$ such that $x\ast b = a$, i.e., there is a second operation, denoted $\bar{\ast}$, such that $(a\ast b)\bar{\ast} b = a$.

\item For any $a, b, c\in X$, $(a\ast b)\ast c = (a\ast c)\ast (b\ast c)$
\end{enumerate}
The ordered pair $(X, \ast)$ denotes a quandle.
\end{definition}

\begin{ex}

\begin{itemize}
\item  The fundamental quandle of a knot \cite{DJoyce, SMatveev} is defined as follows. Given any $($oriented$)$ diagram of $K$, a presentation of the fundamental quandle of $K$ is obtained as follows.  Arcs of the diagram stand for generators and relations are read at each crossing as illustrated in Figure \ref{fig:fig8-KnotQuandle}. The fundamental quandle of the knot is also the \emph{knot quandle}, in this article.
\begin{figure}[!ht]
	\psfrag{a-ast b = c}{\huge$a\ast b = c$}
	\psfrag{b-ast a = d}{\huge$b\ast a = d$}
	\psfrag{b-ast d = c}{\huge$b\ast d = c$}
	\psfrag{a-ast c = d}{\huge$a\ast c = d$}
	\psfrag{a}{\huge$a$}
	\psfrag{b}{\huge$b$}
	\psfrag{c}{\huge$c$}
	\psfrag{d}{\huge$d$}
	\psfrag{M}{\huge$$}
	\psfrag{M'}{\huge$$}
	\centerline{\scalebox{.35}{\includegraphics{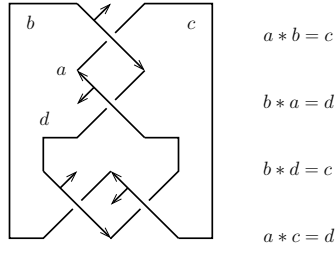}}}
	\caption{Simplifying presentations of the Knot Quandle of the Figure $8$ knot: $(\, a, b, c, d \, |\, a\ast b = c, b\ast a = d, b\ast d = c, a\ast c = d\, )\cong (\, a, b, c \, |\, a\ast b = c, b\ast (b\ast a) = c, a\ast c = b\ast a\, ) \cong (\, a, b \, |\,  b\ast (b\ast a) = a\ast b, a\ast (a\ast b) = b\ast a\, ).$}\label{fig:fig8-KnotQuandle}
\end{figure}
\item Dihedral quandle of order $n$. The underlying set is $\mathbf{Z}_n$ and the operation is, for any $a, b\in \mathbf{Z}_n$, $a\ast b = 2b-a$, mod $n$. Note that in this case, $a\bar{\ast} b = a\ast b$.
\item Linear Alexander quandle, $($LAQ$)$. Given positive integers $m < n$ such that $(n, m)=1$, we let the underlying set be $\mathbf{Z}_n$, and we let the operation be, for  any $a, b\in \mathbf{Z}_n$, $a\ast b = ma+(1-m)b$, mod $n$. We denote such quandles $LAQ(n, m)$.
\item The Alexander quandle. The underlying set is the set of the Laurent polynomials, $\mathbf{Z}[T^{\pm 1}]$. The quandle operation is, for any $a, b\in \mathbf{Z}[T^{\pm 1}]$, $a\ast b = Ta+(1-T)b$.
\end{itemize}
\end{ex}

\begin{theorem}[\cite{DJoyce, SMatveev}]
The fundamental quandle of the knot is a knot invariant $($up to isomorphism$)$.
\end{theorem}
\begin{proof}
Proofs of this fact are found in \cite{DJoyce} and \cite{SMatveev}.
\end{proof}
\begin{definition}[Quandle homomorphism]
Given quandles $(X, \ast)$ and $(X', \ast')$, a quandle homomorphism is a mapping $h : X \longrightarrow X'$ such that, for any $a, b\in X$, $$ h(a\ast b) = h(a)\ast' h(b).$$
\end{definition}
\begin{cor}\label{cor:cardhomo}
Let $K$ be a knot and $X$ a quandle. The number $($which may be infinite$)$ of homomorphisms from the Knot Quandle of $K$ to $X$ is a knot invariant. In this context, $X$ is called the target quandle.
\end{cor}
\begin{proof}
If it were not, there would not be an isomorphism between the knot quandles.
\end{proof}

In the Example \ref{ex:quandlehomo} right below, we keep the notation and terminology of Corollary \ref{cor:cardhomo}. An otherwise arbitrary knot has been fixed; we vary the family of the target quandles.
\begin{ex}\label{ex:quandlehomo}
\begin{enumerate}
\item Let $X$ equal a dihedral quandle of order $m$. Then the homomorphisms are the $m$-colorings of the knot under study.
\item Let $X$ be a linear Alexander Quandle, other than a dihedral quandle. This article is devoted to the study of the homomorphisms with these quandles as targets.
\item Let $X$ be the Alexander quandle. We elaborate on the associated homomorphisms in Subsection \ref{subsect:Alex-Fox-other-colorings}, right below.
\end{enumerate}
\end{ex}

\bigbreak

\subsection{Alexander Colorings}\label{subsect:Alex-Fox-other-colorings}

We now elaborate on the quandle homomorphisms $3$. of  Example \ref{ex:quandlehomo}. These are  the quandle homomorphisms from the knot quandle of a given knot, to the Alexander quandle. This is the quandle whose underlying set is the ring of Laurent polynomials (on the variable $T$), denoted $\Lambda$, i.e., $$\Lambda = \mathbf{Z}[T, T^{-1}], $$ equipped with the operation $$a\ast b = Ta+ (1-T)b \qquad \text{ for any } a, b\in \Lambda .$$

The coloring condition here is depicted in Figure \ref{fig:alexcoloringcondition}.  The procedure and mathematical objects that come up along the way parallel those of the preceding cases. In short, there will be a matrix (here called the Alexander matrix) whose first minor determinant (here called the (1st) Alexander polynomial) controls the existence of non-trivial solutions (colorings); should these solutions exist they belong in the appropriate quotient of $\Lambda$ (modulo the ideal generated by the Alexander polynomial, or one of its factors). More specifically, given a knot, the homomorphisms from its knot quandle to the Alexander quandle are solutions of a system of equations (the coloring conditions) read off the crossings of a given diagram of the knot. This system of equations gives rise to a matrix of coefficients (which is the Alexander matrix). The fact that along each row the non-null entries are $T$, $1-T$, and $-1$, implies this matrix has $0$ determinant; this corresponds to the existence of the trivial (i.e., monochromatic) solutions. The existence of non-trivial solutions corresponds to working on the quotient of $\Lambda$ by the ideal generated by the first minor determinant of the Alexander matrix (i.e., the Alexander polynomial). The Alexander polynomial is a Laurent polynomial on the variable $T$ and is determined up to $\pm T^n$, for any integer $n$; it is independent, up to $\pm T^n$, of the first minor you choose from the Alexander matrix. Further information on the Alexander matrix and the Alexander polynomial(s) can be found in \cite{CFox}, although from a different perspective.
\begin{figure}[!ht]
	\psfrag{a}{\huge$a$}
	\psfrag{b}{\huge$b$}
	\psfrag{c = ma+(1-m)b}{\huge$a\ast b = Ta+(1-T)b$}
	\centerline{\scalebox{.35}{\includegraphics{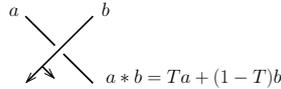}}}
	\caption{The coloring condition at a crossing of the diagram with $\Lambda$ as target quandle.}\label{fig:alexcoloringcondition}
\end{figure}

\begin{figure}[!ht]
	\psfrag{a-ast b = c}{\huge$a\ast b = c \quad : \quad Ta+(1-T)b=c$}
	\psfrag{b-ast a = d}{\huge$b\ast a = d \quad : \quad Tb+(1-T)a=d$}
	\psfrag{b-ast d = c}{\huge$b\ast d = c \quad : \quad Tb+(1-T)d=c$}
	\psfrag{a-ast c = d}{\huge$a\ast c = d \quad : \quad Ta+(1-T)c=d$}
	\psfrag{a}{\huge$a$}
	\psfrag{b}{\huge$b$}
	\psfrag{c}{\huge$c$}
	\psfrag{d}{\huge$d$}
	\psfrag{M}{\huge$$}
	\psfrag{M'}{\huge$$}
	\centerline{\scalebox{.35}{\includegraphics{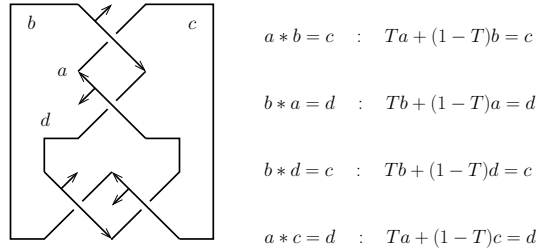}}}
	\caption{The equations for the Alexander colorings of the figure eight knot.}\label{fig:fig8-Alex}
\end{figure}
We now illustrate the procedure with the help of the ``figure eight'' knot, see Figure \ref{fig:fig8-Alex}. Due to the coloring of the diagram by the Alexander quandle, the matrices we associate with the system of equations in Figure \ref{fig:fig8-Alex} are $M$, the Alexander matrix, and $M_1$, one of its first minor matrices:
\begin{align}\label{eq:1}
M=
\begin{pmatrix}
{\bf a} & {\bf b} & {\bf c} & {\bf d}\\\hline
T & 1-T & -1 & 0\\
1-T & T & 0  & -1\\
0 & T & -1 & 1-T\\
T & 0 & 1-T & -1
 \end{pmatrix}
\qquad \qquad
M_1 =\begin{pmatrix}T & 0 & -1 \\
 T & -1 & 1-T\\
 0 & 1-T & -1
 \end{pmatrix}
 \end{align}
and the determinant of $M_1$ is the  Alexander polynomial of the Figure eight knot. This is
\begin{align}\label{eq:2}
\det (M_1) = T\begin{vmatrix}
 -1 & 1-T\\
  1-T & -1
 \end{vmatrix} - \begin{vmatrix}
 T & -1 \\
 0 & 1-T
 \end{vmatrix} = - T(T^2-3T+1) \overset{\cdot}{=} T^2-3T+1
\end{align}
(where $\overset{\cdot}{=}$ means ``equality modulo units of $\Lambda$'', $\pm T^n$'s) so the reduced Alexander polynomial (defined below) of the figure eight knot, $K$,  is $$\Delta^0_K (T) = T^2-3T+1 .$$

The $4\times 4$ matrix in Equation (\ref{eq:1}) is the matrix of the coefficients of the associated linear homogeneous system of equations obtained from collecting the coloring conditions at each crossing of the diagram. The determinant of this matrix is zero since along each row we find exactly one $T$, one $1-T$ and one $-1$ along with $0$'s. This means that we have several solutions. We call them trivial or monochromatic because each of these solutions assign the same Laurent polynomial to the distinct arcs of the diagram. In order to obtain polychromatic (or non-trivial) solutions we have to work over a quotient of $\Lambda$ where the first minor determinant of the coloring matrix vanishes \cite{McDonald}. This first minor determinant (Equation (\ref{eq:2})) is the Alexander polynomial (modulo multiplication by units of $\Lambda$) of the knot at issue  \cite{CFox}.

\begin{definition}
When an Alexander polynomial is not identically zero, the reduced Alexander polynomial is the one which when evaluated at $0$ is defined and positive. For instance, $\Delta^0_K (T) = T^2-3T+1, $
is the reduced Alexander polynomial of the figure eight knot.
\end{definition}

CAVEAT: In this article we only consider knots or links whose Alexander polynomial is not identically $0$.

We now give a formal statement and proof of the fact above.
\begin{prop}
In order to obtain polychromatic i.e., non-trivial, solutions, the system of coloring equations has to be considered over a quotient of $\Lambda$ where the first minor determinant of the coloring matrix $($i.e., Alexander matrix$)$ vanishes.
\end{prop}
\begin{proof}
The proof has two parts. In the first part we prove that the sum of a polychromatic solution with a monochromatic solution is again a (polychromatic) solution, see Figure \ref{fig:addtrivialcoloring}. \begin{figure}[!ht]
	\psfrag{a}{\huge$a$}
	\psfrag{b}{\huge$b$}
	\psfrag{c}{\huge$c = Ta+(1-T)b$}
	\psfrag{a1}{\huge$a + const$}
	\psfrag{b1}{\huge$b + const$}
	\psfrag{c1}{\huge$Ta+(1-T)b + const = c + const$}
	\centerline{\scalebox{.35}{\includegraphics{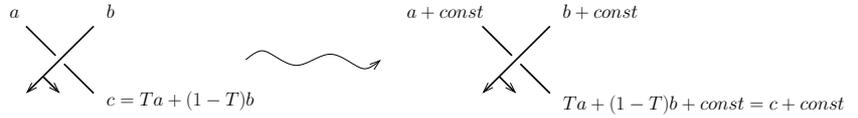}}}
	\caption{The addition of a monochromatic solution (denoted ``$const$'', in the Figure) to any other solution is again a solution, for Alexander quandles: $T(a + const) + (1-T)(b + const) = Ta+(1-T)b + const = c + const$. The Figure checks that if at a crossing the triple $(a, b, c)$ satisfies the coloring condition ($a\ast b= c$), then so does the triple $(a + const, b + const, c + const)$.}\label{fig:addtrivialcoloring}
\end{figure}

Now for the second part. Let $M$ denote the coloring matrix and assume $\vec{V}$ is a polychromatic solution, so that the equation $$M \vec{V} = \vec{0}$$ holds (where juxtaposition on the left-hand side of the equation denotes matrix multiplication and $\vec{0}$ is a column vector consisting of $0$'s). We now add an appropriate monochromatic solution to $\vec{V}$ so that the $i$-th component of the new solution is $0$; we denote it $\vec{V}_i$. We remark that $\vec{V}_i$ is again a polychromatic solution. We now let $\vec{v}_i$ denote the column vector obtained from $\vec{V}_i$ by removing the $i$-th component; $\vec{v_i}$ is also polychromatic. Furthermore, Let $M_i$ be the matrix obtained from $M$ by removing the $i$-th column and the $i$-th row (although we could have removed any row). Then \cite{McDonald} $$\vec{0} = M_i \vec{v}_i \quad \Longleftrightarrow \quad \vec{0} = Adj(M_i) M_i \vec{v}_i = (\det M_i)\,  \vec{v}_i , $$   where $Adj(M_i)$ denotes the adjoint matrix to $M_i$. Since $\vec{v}_i$ is a polychromatic solution then $\det M_i$ has to be $0$.  We bring that about by working mod $\det M_i$ i.e., by working on the quotient of $\Lambda$ by the ideal generated by $\det M_i$ (or any factor of it, should it exist).

 The proof is complete.
 \end{proof}

  We remark that $\det M_i$ is the Alexander polynomial of the knot at issue \cite{CFox}; it is independent of the pair row-column removed above, modulo multiplication by units. In this way, the Alexander polynomial is the means to obtain non-trivial solutions for the knot under study.

Furthermore, the Alexander polynomial is a knot invariant \cite{CFox} so the considerations above do not depend on the knot diagram being used.

Figure \ref{fig:AlexTrefoil} renders another perspective on how to retrieve the Alexander polynomial, this time for the Trefoil knot. Note that somehow we are doing a Gaussian elimination in order to calculate the relevant determinant.

\begin{figure}[!ht]
	\psfrag{a}{\huge$a$}
	\psfrag{b}{\huge$b$}
	\psfrag{c}{\huge$a\ast b = Ta+(1-T)b$}
	\psfrag{d}{\huge$b\ast (a\ast b) =  (T^2-T+1)(b-a) + a (= a)$}
	\psfrag{e}{\huge$(a\ast b)\ast (b\ast (a\ast b)) = T(T^2-T+1)(a-b) + b (= b) $}
	\centerline{\scalebox{.4}{\includegraphics{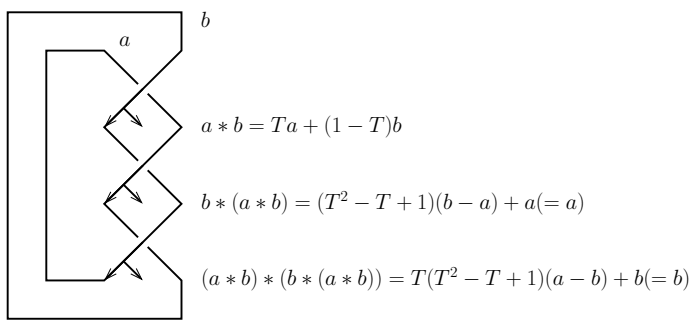}}}
	\caption{Retrieving the Alexander polynomial for the trefoil. The equalities inside parenthesis indicate that the resulting equation follows from the fact the arcs at stake had already been assigned colors. It follows that the Alexander polynomial is $T^2-T+1$ modulo units.}\label{fig:AlexTrefoil}
\end{figure}

\begin{figure}[!ht]
	\psfrag{a-ast b = c}{\huge$a\ast b = c \quad : \quad (-1)a+(1-(-1))b=c$}
	\psfrag{b-ast a = d}{\huge$b\ast a = d \quad : \quad (-1)b+(1-(-1))a=d$}
	\psfrag{b-ast d = c}{\huge$b\ast d = c \quad : \quad (-1)b+(1-(-1))d=c$}
	\psfrag{a-ast c = d}{\huge$a\ast c = d \quad : \quad (-1)a+(1-(-1))c=d$}
	\psfrag{a}{\huge$a$}
	\psfrag{b}{\huge$b$}
	\psfrag{c}{\huge$c$}
	\psfrag{d}{\huge$d$}
	\psfrag{M}{\huge$$}
	\psfrag{M'}{\huge$$}
	\centerline{\scalebox{.35}{\includegraphics{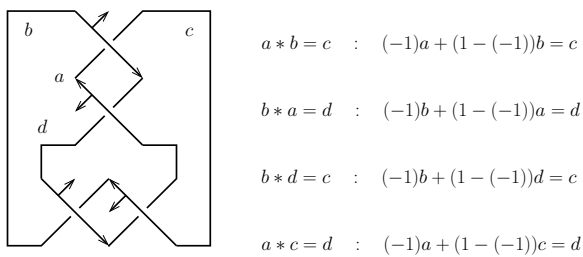}}}
	\caption{The equations for the Fox colorings of the Figure $8$ knot; coloring condition: $a\ast b = 2b-a$.}\label{fig:fig8-Fox}
\end{figure}
Let us now look at some particular cases. Suppose we set $T = -1$; then the underlying set of the target quandle is the integers, $\mathbf{Z}$, and the quandle operation is $a\ast b = 2b-a$. The Alexander polynomial evaluated at $T = -1$  yields the so-called knot determinant and its absolute value yields the generator of the ideal associated with the nontrivial solutions i.e., the non-trivial colorings. This is the context of the so-called Fox colorings - see Figure \ref{fig:fig8-Fox} for an example with the figure eight knot.

\begin{figure}[!ht]
	\psfrag{a-ast b = c}{\huge$a\ast b = c \quad : \quad ma+(1-m)b=c$}
	\psfrag{b-ast a = d}{\huge$b\ast a = d \quad : \quad mb+(1-m)a=d$}
	\psfrag{b-ast d = c}{\huge$b\ast d = c \quad : \quad mb+(1-m)d=c$}
	\psfrag{a-ast c = d}{\huge$a\ast c = d \quad : \quad ma+(1-m)c=d$}
	\psfrag{a}{\huge$a$}
	\psfrag{b}{\huge$b$}
	\psfrag{c}{\huge$c$}
	\psfrag{d}{\huge$d$}
	\psfrag{M}{\huge$$}
	\psfrag{M'}{\huge$$}
	\centerline{\scalebox{.35}{\includegraphics{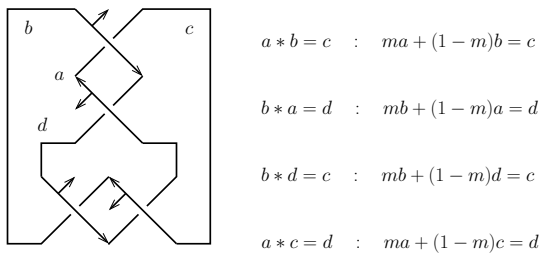}}}
	\caption{The equations for the colorings beyond Fox of the Figure $8$ knot; coloring condition: $a\ast b = ma+(1-m)b$.}\label{fig:fig8-LAQ}
\end{figure}

We finally consider the case where we evaluate $T$ at other integer values, see Figure \ref{fig:fig8-LAQ}. This gives rise to the colorings beyond Fox, the topic of the current article. We elaborate on this in Subsection \ref{subsect:metho} right below.

\bigbreak

\subsection{Methodology for the colorings using the other Alexander quandles.}\label{subsect:metho}

As in the case of Fox colorings, where we implicitly choose a dihedral quandle for target quandle, if we choose another linear Alexander quandle ($a\ast b := ma +(1-m)b$) for target quandle, we will obtain a system of linear homogeneous equations over the integers whose solutions yield the colorings of the knot under study with respect to the chosen linear Alexander quandle. We remark that in the case of the Fox colorings (i.e, dihedral quandles for target quandles) the $m$ parameter is already chosen; it is $m=-1$.

The  \emph{linear Alexander quandles} constitute another family of quandles that contains the dihedral quandles. Each linear Alexander quandle is indexed on the modulus, say $n$ (which is its order) and a second integer parameter, say $m$, such that  $(n, m) = 1$. We recover the dihedral quandles by setting $m=-1$,  for each modulus $n$. The operation here is
\[
a \ast b = ma + (1-m)b \qquad \text{ mod } n
\]
The requirement $(n, m) = 1$ guarantees the right-invertibility of the operation i.e, the second axiom for quandles. We warn the reader that different choices of $m$ for the same modulus $n$ may give rise to isomorphic linear Alexander quandles. Notwithstanding, if we choose a prime modulus $n$, then each of the integers $1<m<n$ comply with $(n, m) = 1$. Moreover, these are the only linear Alexander quandles with the indicated prime order. Thus, for each odd prime $n$, each $m$ such that $1 < m < n$ corresponds to a non-isomorphic quandle  \cite{Nelson}. In the current article we will always choose a prime modulus, unless otherwise stated. Furthermore, once $m$ is chosen, the prime modulus, say $p$, should be a factor of the Alexander polynomial of the knot at stake and evaluated at $m$, leaning on the results of Subsection \ref{subsect:Alex-Fox-other-colorings}. In this article we will prefer this prime factor to be the reduced Alexander polynomial evaluated at the given $m$. Table \ref{Ta:Delta^0_K (m) = m^2-3m+1} shows the values of the reduced Alexander polynomial of the figure eight knot at the first  few positive integers.

\begin{definition}
Given a knot $K$, we call $m$-determinant of $K$, notation $m-\det_K$, the value of the reduced Alexander polynomial of $K$ evaluated at $m$. We will  write often $m-\det$ for $m-\det_K$ when it is clear from context which $K$ we are referring to.
\end{definition}

\begin{table}[h!]
\begin{center}
\begin{tabular}{| c ||  c |   c  |   c  |   c  |   c  |   c  |   c  |   c  |   c  |   c  |   c  |   c | }\hline
$m$                            & $2$ & $3$ & $4$ & $5$ & $6$ & $7$ & $8$ & $9$  & $10$ & $11$ & $12$ & $13$ \\ \hline
$\Delta^{0}_K (m) = m^2-3m+1 $  &  $-1$    & $4$   & $5$   & $11$  &  $19$    & $29$  &  $41$   & $55$   &  $71$       & $89$   &  $109$   & $131$    \\ \hline
\end{tabular}
\caption{$\Delta^0_K (m) ( = m-\det_K) = m^2-3m+1$, the reduced Alexander polynomial of the figure-eight knot, as a function of $m$ for the first few positive integers.}
\label{Ta:Delta^0_K (m) = m^2-3m+1}
\end{center}
\end{table}

\begin{definition}\label{def:fox-n-m}
For positive integer $n>2$ and integer  $m$ such that $a\ast b = ma +(1-m)b$, mod $n$ is a quandle over $\mathbf{Z}_n$, and knot $K$, we call   $(n, m)$-{\bf colorings} of $K$, the homomorphisms from the knot quandle of $K$ to the linear Alexander quandle of order $n$ and  parameter $m$. We note that, at each crossing,  there is a preferred under-arc which receives the product, in the case of  $(n, m)$-\emph{colorings} for $m\neq -1$ mod $n$ - see Figure \ref{fig:m-xing}.
\end{definition}

\begin{figure}[!ht]
	\psfrag{a}{\huge$a$}
	\psfrag{b}{\huge$b$}
	\psfrag{c = ma+(1-m)b}{\huge$a\ast b = ma+(1-m)b$}
	\centerline{\scalebox{.35}{\includegraphics{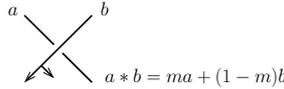}}}
	\caption{The coloring condition at a crossing for a \emph{Fox} $n, m$-\emph{coloring}.}\label{fig:m-xing}
\end{figure}

We leave it to the reader to show that for $m=-1$ no orientation of the diagram is needed. We conclude this Subsection with preliminary results on the minimum number of colors, Proposition \ref{prop:min3} and Corollary \ref{cor:min3}, below.
\begin{definition}\label{def:splitlink}
Let $K$ be a link $($multiple component knot$)$. $K$ is said to be a {\bf split link} if there exist disjoint balls in $3$-space and a deformation of $K$ such that some components of deformed $K$ are in one of the balls while the other components are in the other ball.

If $K$ is not a split link, $K$ is said to be a {\bf non-split link}.

We  regard a one-component knot as a non-split link.
\end{definition}
\begin{prop}\label{prop:min3}
If a knot or non-split link admits a non-trivial  $(p, m)$-\emph{coloring} with $p$ odd prime and $1<m<p$, then at least three colors are needed to assemble such a coloring.
\end{prop}
\begin{proof}
If there is a non-trivial coloring, there has to be at least one crossing where  two distinct colors meet (call these colors $a$ and $b$). Otherwise the link would be split. We will show that there has to be a third color in order for the  $(p, m)$-\emph{coloring} condition to be satisfied at this crossing. We prove this by showing that the existence of only two colors in this crossing implies they are equal.
\begin{itemize}
\item Assume the over-arc is assigned the same color, say $a$, as one of the under-arcs.
\begin{enumerate}
\item Assume that under-arc is the one that does not receive the product. Then the coloring condition yields the color on the arc that receives the product. This is $ma+(1-m)a = a$ which implies the three colors meeting at this crossing are the same. This does not comply with the hypothesis;
\item Assume now the over-arc and the under-arc which receives the product are both colored $a$. Then $a=mx+(1-m)a$ where $x$ is the yet unknown color of the other under-arc. The equation simplifies to $mx=ma$ and since $p$ is prime and $1<m<p$, then $m$ is invertible so the equation further simplifies to $x=a$ which again is a contradiction.
\end{enumerate}
\item Then the color on the over-arc has to be distinct from the colors on the under-arcs. Forcing the existence of two colors, this implies that the over-arc is colored $b$ and the under-arcs are colored $a$. Applying the coloring condition we obtain $a=ma+(1-m)b$ which simplifies to $(m-1)(b-a)=0$ mod $p$ which implies that $b=a$, since $p$ is prime and $0<m-1<p$. This is again a contradiction.
\end{itemize}The proof is complete.
\end{proof}
\begin{cor}\label{cor:min3}
Keeping the conditions of Proposition \ref{prop:min3} and setting $p=3$, then it takes exactly $3$ distinct colors to assemble such a non-trivial coloring.
\end{cor}
\begin{proof}
In this case, the least number of colors expected (Proposition \ref{prop:min3}) and the  number of colors available both equal $3$. The proof is complete.
\end{proof}

This article is devoted mainly to generalizing the results on minimum number of colors from dihedral quandles to linear Alexander quandles and to trying to expose differences that may occur. Proposition \ref{prop:min3} is an example of a phenomenon already known for Fox colorings. An instance which does no occur for Fox colorings is illustrated in Subsection \ref{subsect:anti-KH behavior}.

{\bf Remark} We  will assume that the knot under study will admit non-trivial $m$-colorings for a given prime $p$ standing for the value of the reduced Alexander polynomial of the knot. The parameter $m$ will eventually turn out to depend on $p$. For instance for the trefoil knot, $m^2-m+1 = p$, which is realized for $p=3$ and $m=-1$.

\section{The palette graph}\label{sect:palette}

In this Section we study the coloring structure where we use for target quandles linear Alexander quandles of type $(p, m)$ - $p$ is an odd prime and $m$ is  an integer such that, preferably, but not necessarily, $1<m<p$. We estimate the determinants of the coloring matrices for these target quandles and we eventually arrive at a lower bound for the number of colors (Theorem \ref{thm:satohmin}). This is useful because when removing colors from a coloring, if we arrive at the corresponding lower bound, as dictated by Theorem \ref{thm:satohmin}, we will know we have reached the minimum number of colors. Figure \ref{fig:6-1-m-3-reducing} provides such an example although the lower bound attained is not provided by Theorem \ref{thm:satohmin} but by Proposition \ref{prop:min3}. The work in this Section is a generalization of \cite{NNS}.

We begin by examining the matrices that are of the sort that arise as coloring matrices for a knot, as we have described them in the previous Section. For instance, the coloring matrix associated with Figure \ref{fig:fig8-LAQ} is
\begin{align}
M=
\begin{pmatrix}
-m & m-1 & 1 & 0\\
m-1 & -m & 0  & 1\\
0 & -m & 1 & m-1\\
-m & 0 & m-1 & 1
 \end{pmatrix}
\end{align}
This analysis leads us to estimate the minimum number of colors from below and obtain the result in Theorem \ref{thm:satohmin} - see proof in Theorem \ref{thm:satohminbis} at the end of this Section.

\begin{definition}\label{def:MatNm}$[\emph{$Mat_N^m$ \text{ matrices}}]$
Let $m$ be an integer and $N$ a positive integer. Let $Mat_N^m$ be the set of $N\times N$ matrices over the integers and such that
\begin{itemize}
\item each row contains at most one $-m$, at most one $m-1$, and at most one $1$, all other entries being $0$, should there be any more entries in the row.
\end{itemize}
\end{definition}

\begin{lem}\label{lem:MatNm}
If $X\in Mat_N^m$ then $|\det X | \leq M^N$, where $M := \max\{ |m|, |m-1| \}$.
\end{lem}
\begin{proof}
The proof is by induction on $N$. If $N=1$ then the only entry in $X$ is either $0$, $-m$, $m-1$ or $1$ and the inequality holds. Now for the induction step. For positive $m$, we split the proof into three instances.
\begin{enumerate}
\item If $X$ has a row without a $-m$, then the minor expansion along this row yields
\[
| \det X | \leq  1\cdot M^{N-1} + (m-1)\cdot M^{N-1}  \leq M^N
\]
\item  If $X$ has a row with exactly one $-m$, apart from the $0$'s, then
\[
| \det X | \leq |-m|\cdot M^{N-1} \leq M^N
\]
\item Otherwise, any row of $X$ has a $-m$ and at least one of $m-1$ or $1$. If need be by swapping columns, we let  its $(1, 1)$ entry be $-m$, and still call this matrix $X$. We let $\vec{v}_j$ be the $j$-th column of $X$. Then the following matrix satisfies $1$. above, by way of its first row:
\[
Y = \bigg(  -\sum_{j=1}^N \vec{v}_j ,\,  \vec{v}_2, \,  \vec{v}_3, \dots , \, \vec{v}_N \bigg)
\]
Then,
\[
|\det X | = |\det Y| \leq M^N.
\]
\end{enumerate}

For negative $m$ the argument is analogous to the one above but $m-1$ takes up the role of $-m$. In this way, the splitting into the three instances is now the following.
\begin{enumerate}
\item If $X$ has a row without an $m-1$, then the minor expansion along this row yields
\[
| \det X | \leq  1\cdot M^{N-1} + (-m)\cdot M^{N-1}  \leq M^N
\]
\item  If $X$ has a row with exactly one $m-1$, apart from the $0$'s, then
\[
| \det X | \leq |m-1|\cdot M^{N-1} \leq M^N
\]
\item Otherwise, any row of $X$ has an $m-1$  and at least one of $-m$ or $1$.  If need be by swapping columns, we let  its $(1, 1)$ entry be $m-1$, and still call this matrix $X$. We let $\vec{v}_j$ be the $j$-th column of $X$. Then the following matrix satisfies $1$. above, by way of its first row:
\[
Y = \bigg(  -\sum_{j=1}^N \vec{v}_j ,\,  \vec{v}_2, \,  \vec{v}_3, \dots , \, \vec{v}_N \bigg)
\]
Then,
\[
|\det X | = |\det Y| \leq M^N.
\]
\end{enumerate}
The proof is complete.
\end{proof}

\begin{definition}\label{def:palettegraph}$[$$(n, m)$-palette graphs and $(n, m)$-palette graphs of diagrams.$]$

Let $n$ be an integer greater than $2$, and $m$ an integer such that $a\ast b := ma+(1-m)b$ defines a quandle over $\mathbf{Z}_n$. Let $S$ be a subset of $\mathbf{Z}_n$. The {\rm\bf $(n, m)$-palette graph} of $S$ is a directed graph whose vertices are the elements of $S$ and whose directed edges are defined as follows. For each  $s_1, s_3 \in S$,  there is an edge from $s_1$ to $s_3$ if there exists an $s_2 \in S$ such that $m s_1 + (1-m) s_2 = s_3$, mod $n$. This edge is labeled $s_2$. A sequence $(s_1, s_2, s_3)$ from $S$ as above is called a {\rm\bf local $(n, m)$-coloring of a potential crossing mod $n$} - or simply a {\rm\bf local coloring}, when the context is clear.

Since a diagram equipped with a non-trivial $(n, m)$-coloring $(n$ and $m$ as above$)$ uses a subset $S$ of colors mod $n$, we call the {\rm\bf $(n, m)$-palette graph of the coloring}, the $(n, m)$-palette graph whose vertices correspond to the colors used in this coloring and whose directed edges $($together with the source vertex and the target vertex$)$ correspond to the distinct solutions of the coloring condition present in the coloring under study. A sequence $(s_1, s_2, s_3)$ from $S$ such that $s_1$ and $s_3$ are the colors of the under-arcs at a crossing whose over-arc is colored $s_2$, and $m s_1 + (1-m) s_2 = s_3$, mod $n$ is called a {\rm\bf local $(n, m)$-coloring of a crossing mod $n$} - or simply a {\rm\bf local coloring}, when the context is clear. If need be, we will add broken lines to the graph for those edges that belong to the palette graph of $S$ but do not correspond to colors on over-arcs in the coloring under study i.e., do not belong to the $(n, m)$-palette graph of the coloring.

In general, we consider the palette graph of a coloring, usually drawn next to the diagram, without the broken lines. Also, ``palette graph'' will mean the palette graph of a coloring, unless explicitly stated.

Finally, for finite $S$ and for any diagram, we remark that both notions $($palette graph and palette graph of coloring$)$ make sense for $n=0$ i.e., upon replacement of $\mathbf{Z}_n$ for $\mathbf{Z}$, above - see Figure \ref{fig:6-1}, for example. The corresponding colorings are called \rm{ integral colorings}. This instance will also be considered in the sequel, namely in Lemma \ref{lem:dets} and Corollary \ref{cor:integer-roots-Alexpoly}.
\end{definition}

In Figure \ref{fig:8-7-Fox} we give an illustrative example for this definition. We remark that the graphs are not directed in Figure \ref{fig:8-7-Fox} since the quandle operation does not have a preferred product for dihedral quandles. On the other hand, in Figure \ref{fig:6-3} the graphs are directed because here the quandle at stake is a linear Alexander quandle other than a dihedral quandle.

\begin{nota}
The following remark is in order here. In the sequel, the figures pertaining to colorings of knots by linear Alexander quandles display a boxed pair of integers in the top left. This pair of integers stands for the pair $(p, m)$ which characterizes the linear Alexander quandle which is being used for target quandle. $p$ will be an odd prime, or $0$, and is the value of the reduced Alexander polynomial of the knot under study, evaluated at the corresponding $m$. For Fox colorings $($dihedral quandle for target quandle$)$, $m=-1$; for all other linear Alexander quandles for target quandles, $m$ will satisfy $1< m < p$.

For example, in Figure \ref{fig:8-7-Fox}, the boxed ``$23, -1$'', stands for $|\Delta^0_{8_7}(-1)|=23$.
\end{nota}

\begin{figure}[!ht]
	\psfrag{8-7}{\huge$\mathbf{8_7}$}
	\psfrag{0}{\huge$0$}
	\psfrag{1}{\huge$1$}
	\psfrag{2}{\huge$2$}
	\psfrag{3}{\huge$3$}
	\psfrag{6}{\huge$6$}
	\psfrag{11}{\huge$11$}
	\psfrag{16}{\huge$16$}
	\psfrag{21}{\huge$21$}
	\psfrag{a}{\huge$a$}
	\psfrag{b}{\huge$b$}
	\psfrag{Fox Colorings}{\huge$\textbf{Fox Colorings}$}
	\psfrag{c=2b-a}{\huge$c=2b-a$}
	\psfrag{23, -1}{\huge$\mathbf{23, -1}$}
	\psfrag{m=2}{\huge$\mathbf{m=2}$}
	\centerline{\scalebox{.35}{\includegraphics{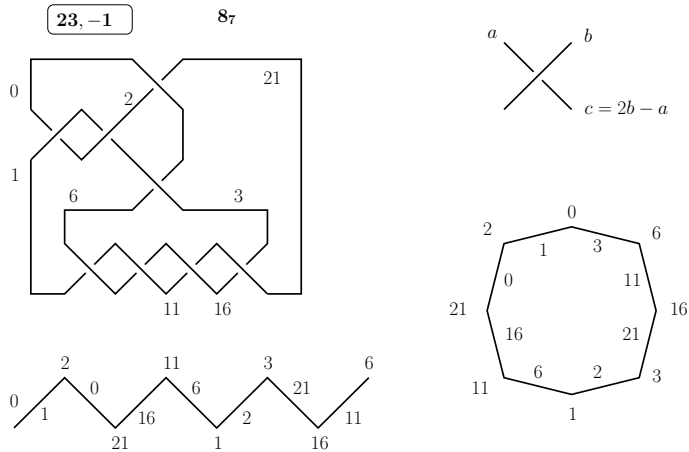}}}
	\caption{The knot $8_7$ whose determinant is $23$.   On the left-hand side, a diagram  equipped with a non-trivial $23$-coloring. The graph depicted in the bottom on the right is the palette graph. The other graph is a spanning tree of the palette graph.}\label{fig:8-7-Fox}
\end{figure}

\begin{lem}\label{lem:palettegraphknotsconnected}
For knots $($1-component links$)$, the palette graph is connected.
\end{lem}
\begin{proof}Walk along the oriented diagram (equipped with a non-trivial (p, m)-coloring) starting at a given under-arc. Mirror that walk in the palette graph of the coloring, starting at the vertex whose color is the color of the under-arc where you started in the diagram. If you are going under a monochromatic crossing in the diagram, you stay in the palette graph, at the vertex whose color is the color of the monochromatic crossing at stake. Otherwise move, in the palette graph, to the vertex whose color is the color of the under-arc on the other side of the crossing you just went under in the diagram. Since there is only one component in the diagram, you must visit all colors in the diagram, thus visiting all colors in the palette graph. Hence there is a path in the palette graph connecting any two colors (i.e., vertices).
\end{proof}
\begin{lem}\label{lem:consequence}
For each circuit in the palette graph of a coloring there is an equation, $E$, (represented by one of the edges of the circuit) which is a consequence of the other equations (represented by the other edges in the circuit). Thus, the coloring at stake is still a solution of the system of equations obtained by removal of equation $E$ from the system of coloring conditions associated with the colored diagram at stake. \end{lem}
\begin{proof}
Pick a circuit in the palette graph with say n vertices, $v_1, ..., v_n$, where $v_i$ is succeeded by $v_{i+1}$ as you go along the circuit and $v_{n+1}$ is $v_1$. Let $w_{i, i+1}$ be the color of the edge connecting $v_i$ to $v_{i+1}$. Then
\begin{align*}
&m_0v_1+(1-m_0)w_{n, 1}= v_n = m_{n-1}v_{n-1} + (1-m_{n-1})w_{n-1, n} = \\
& = m_{n-1}[m_{n-2}v_{n-2} + (1-m_{n-2})w_{n-2, n-1}] + (1-m_{n-1})w_{n-1, n} = \cdots =\\
& = \bigg(\prod_{i=1}^{n-1}m_{n-i}\bigg) v_1 +  \bigg(\prod_{i=2}^{n-1}m_{n-i}\bigg) (1-m_2) w_{3, 2} + \cdots + m_{n-1} (1 - m_{n-2})w_{n-2, n-1} + (1-m_{n-1})w_{n-1, n}
\end{align*}
where the $m_i$'s mean  either $m$ or $m^{-1}$, according to the coloring.
\end{proof}

In this way, we are interested in the graph obtained from the palette graph by taking the spanning tree of each of its components; we call this graph the {\it spanning forest} of the palette graph. If we are working with a knot (i.e., a 1-component link), the spanning forest of its palette graph is a spanning tree, thanks to Lemma \ref{lem:palettegraphknotsconnected}.
\begin{figure}[!ht]
	\psfrag{0}{\huge$0$}
	\psfrag{1}{\huge$1$}
	\psfrag{2}{\huge$2$}
	\psfrag{3}{\huge$3$}
	\psfrag{4}{\huge$4$}
	\psfrag{5}{\huge$5$}
	\psfrag{a}{\huge$a$}
	\psfrag{b}{\huge$b$}
	\psfrag{2a-b}{\huge$c=2a-b$}
	\psfrag{6-3}{\huge$\mathbf{6_3}$}
	\psfrag{p=7}{\huge$\qquad \mathbf{p=7}$}
	\psfrag{7, 2}{\huge$\mathbf{7, 2}$}
	\centerline{\scalebox{.35}{\includegraphics{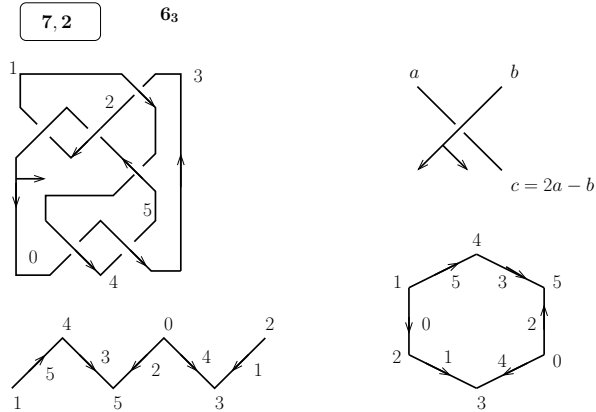}}}
	\caption{The knot $6_3$ whose $m$-determinant is $m^4-3m^3+5m^2-3m+1$. For $m=2$ we obtain $2-\det = 7$.  On the left-hand side, a diagram  equipped with a non-trivial $(2,7)$-coloring. The graph depicted at the bottom right is the palette graph of the coloring. The other graph is the spanning forest (in fact, spanning tree) of the former graph.}\label{fig:6-3}
\end{figure}

% whose $r$-th tree has vertices $c^r_1, c^r_2, \dots , c^r_{k_{r}}$ and edges $e^r_1, e^r_2, \dots , e^r_{k_{r}-1}$

\begin{definition}\label{def:madjacency matrix}
Keeping the notation and terminology from Definition \ref{def:palettegraph}, let $G$ be an $(n, m)$-palette graph with $k$ colors. Let $F$ be a spanning forest of $G$. In the sequel, the relationships ``vertex-variable-column'' and ``edge-equation-row'' should be born in mind. We now refer the reader to Figure \ref{fig:palettegraph}. We call the {\rm\bf $(n, m)$-adjacency matrix} of $F$, the matrix, $A$, whose $(i, j)$-entry is
\begin{itemize}
\item $1$, if edge $e_i$ ends at vertex $c_j$,
\item $-m$, if edge $e_i$ begins at vertex $c_j$,
\item $m-1$, if edge $e_i$ begins at vertex $c_l$ and ends at vertex $c_k$ such that $m\cdot c_l + (1-m)\cdot c_j = c_k$, mod $n$ $($i.e., if $e_i$ is labeled with color $c_j)$,
\item $0$, otherwise.
\end{itemize}
\end{definition}
\begin{figure}[!ht]
	\psfrag{dots}{\huge$\cdots$}
	\psfrag{a}{\huge$a$}
	\psfrag{b}{\huge$b$}
	\psfrag{c}{\huge$c = ma+(1-m)b = a\ast b$}
	\psfrag{a1}{\huge$c_s$}
	\psfrag{b1}{\huge$e_r \leftarrow c_u$}
	\psfrag{c1}{\huge$c_t = mc_s+(1-m)c_u$}
	\psfrag{ei}{\huge$e_r$}
	\psfrag{cj}{\huge$c_s$}
	\psfrag{ck}{\huge$c_t$}
	\psfrag{cl}{\huge$c_u$}
	\psfrag{1}{\huge$1$}
	\psfrag{-m}{\huge$-m$}
	\psfrag{m-1}{\huge$m-1$}
	\centerline{\scalebox{.35}{\includegraphics{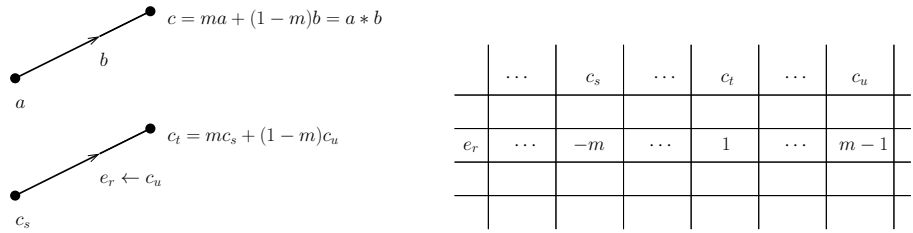}}}
	\caption{Aide-m\'emoire for the adjacency matrix.}\label{fig:palettegraph}
\end{figure}

%We recall that we consider only the knots and links  whose Alexander polynomial is not identically $0$.

\begin{lem}\label{lem:dets}
%We keep the notation and terminology from Definitions \ref{def:palettegraph} and \ref{def:madjacency matrix} but we let $n$ be an odd prime, henceforth denoted $p$.

Let $p$ be an odd prime and $m$ be an integer. Let $K$ be a knot i.e., a $1$-component link, admitting non-trivial $(p, m)$-colorings.

Let $G$ be the $(p, m)$-palette graph of such a coloring, let $A$ be its $(p, m)$-adjacency matrix, and $A_j$ be the square $(k-1)\times (k-1)$ matrix, obtained by deletion of the $j$ column from $A$.  Then,
\begin{enumerate}
\item either $\det A_j = 0$ or $\det A_j$ is divisible by $p$; and
\item $\det A_j = \pm 1$, mod $|m-1|$.
\end{enumerate}
\end{lem}
\begin{proof}
We remark that knot $6_1$ is an example where $\det A_j = 0$, for $m=2$, see Figure \ref{fig:6-1}.
%Furthermore, there are links (with more than one components) whose Alexander polynomial is identically $0$, \cite{Cochran-DS}.

\begin{enumerate}
\item Assume $\det A_j \neq 0$ over the integers. There are two independent solutions for the system of homogeneous linear equations (mod $p$) whose coefficient matrix is $A$, $\vec{x} = (1, 1, \dots \, 1)^T$ (the ubiquitous trivial coloring) and $\vec{x} = (c_1, c_2, \dots , c_k )^T$ (the non-trivial $(p, m)$-coloring the palette graph of the statement stems from). Then the rank of $A$ is at most $k-2$. Thus, $\det A_j = 0$ mod $p$, since rank of $A_j$ is at most $k-1$, mod $p$.

\item First note that the expression is trivially true for $m=2$, since  any two integers are equivalent, mod $1$. Assume now $m\neq 2$ and consider the matrix $A_j$ with the entries read mod $|m-1|$; call it $A_j'$. The entries of $A_j'$ are $0$'s, $1$'s, and $-1$'s. Specifically, the entry $(i, i')$ is $1$ if edge $e_i$ is directed to vertex $c_{i'}$, it is $-1$ if edge $e_i$ stems from vertex $c_{i'}$, and $0$ otherwise. Also note that the entries in $A_j'$ correspond to a ''subgraph'', $T'$, obtained from the spanning tree $T$ of $G$, by removing the vertex corresponding to column $j$. There is then a unique bijection, call it $\sigma$, mapping each edge in $T'$ to a vertex in $T'$ incident to that edge. Note that ``incident'' here means that the vertex may either play the role of source or of sink for the edge at stake. The existence and uniqueness of $\sigma$ is stated and proved below in Claim \ref{cl:cl}.  Then, the only non-null summand of $\det A_j'$ is the one associated with $\sigma$. Thus,
\begin{align}\label{eqn:det}
\pm 1  = a_{1,\, \sigma (1)} a_{2,\, \sigma (2)} \cdots a_{k-1,\, \sigma (k-1)} = \det A_j' = \det A_j \quad \text{ mod } |m-1|.
\end{align}
The proof of the following Claim completes the proof of this Lemma.
\end{enumerate}
\begin{cl}\label{cl:cl}
Let $n$ be an integer greater than $1$ and let $T_n$ denote a tree on $n$ vertices $($and consequently $n-1$ edges$)$. Then, upon removal of one vertex from $T_n$, say $v$, there is a unique bijection from the edges of $T_n$ to $T_n\setminus \{ v \}$ such that each edge is mapped to the vertex incident to it.
\end{cl}
\begin{proof}
The proof is by induction. It is clearly true for $n=2$. Assuming it is true for all positive integers up to a given $n$, consider the $n+1$ instance. If we remove a vertex, call it $v$, from $T_{n+1}$ which has more than one edge incident to it, then we obtain a finite number of disconnected trees, each one satisfying the induction hypothesis. Putting together the bijections for each of these trees, we obtain the bijection for $T_{n+1}\setminus \{ v \}$. If we remove a vertex, call it $V$, which has only one edge incident to it, call it $E$, let us then remove $E$ and let us also remove the other vertex incident to it, call it $U$. The resulting tree now satisfies the induction hypothesis and so there is a unique bijection from edges to vertices such that each edge is mapped to a vertex incident to it. We now augment this bijection by sending $E$ to $U$, thus obtaining the desired bijection for $T_{n+1}\setminus \{  V  \}$. The proof of Claim \ref{cl:cl} is complete.
\end{proof}
The proof of Lemma \ref{lem:dets} is complete.
\end{proof}

%We remark that for $m=2$, Equation (\ref{eqn:det}) equals $0$ mod $|m-1|$. We leave it as it is to encompass all possible values of $m$.

\begin{cor}
If $m \neq 2$ mod $p$, then $\det A_j = \pm 1 + l(m-1)$ for some integer $l$. If $m = 2$ mod $p$, then $\det A_j$ may be zero. This occurs for the knot $6_1$.
\end{cor}
\begin{proof}
According to Lemma \ref{lem:dets}, Equation \ref{eqn:det}, $\det A_j = \pm 1 + l(m-1)$ for some integer $l$. Then, the only way of making $\det A_j$ equal to $0$ is to set $m=2$, $l=\mp 1$ which yields $\det A_j = \pm 1$, mod $m-1$. This occurs for the knot $6_1$ (see Figure \ref{fig:6-1}). Note that it may also not be $0$ (see Figure \ref{fig:8-7-m-p}).
\end{proof}

\begin{cor}\label{cor:integer-roots-Alexpoly}
The only integer root of any reduced Alexander polynomial of a knot $($i.e., $1$-component link$)$ is $2$.
\end{cor}
\begin{proof}
Suppose $m$ is an integer  root of the reduced Alexander polynomial of $K$ i.e., $\Delta^0_K(m)=0$. Note that $\Delta^0_K(m)$ is the minor determinant of the coloring matrix of $K$, with coloring condition $c=ma+(1-m)b \Leftrightarrow ma + (1-m)b - c = 0$. This coloring matrix is an integer matrix and along each row there is exactly one $m$, one $1-m$ and one $-1$, and otherwise $0$'s. Hence the  determinant of the coloring matrix is $0$. Its Smith Normal Form  has therefore at least two $0$'s along the diagonal; one because the coloring matrix it stems from has $0$ determinant, and the other one because its first minor  determinant is $0$, $\Delta^0_K(m)=0$ ($m$ is a root). There is thus at least one non-trivial $(p, m)$-coloring of the diagram, for any odd prime $p$ i.e., a non-trivial integral coloring. Consider the palette graph of this non-trivial coloring, consider a spanning forest of this palette graph and obtain its $m$-adjacency matrix, $A$. Then, with the notation of Lemma \ref{lem:dets}, $$0 = \det A_j \qquad \text{ and } \qquad \det A_j = \pm 1 \quad (\text{mod }|m-1|).$$ But this  is true if and only if $m=2$. The proof is complete.
\end{proof}

\bigbreak

We note that the reduced Alexander polynomial of the Hopf link equals $0$ at $1$.

\bigbreak

We are now ready to prove Theorem \ref{thm:satohmin}. We state it here again for the reader's convenience. We remark again that we exclude from our considerations knots or links whose Alexander polynomial is identically $0$.

\begin{theorem}\label{thm:satohminbis}
Let $K$ be a knot i.e., a $1$-component link and $p$ be an odd prime. Let $m$ be an integer such that $K$ admits non-trivial $(p, m)$-colorings mod $p$. If $m \neq 2$ $($or $m=2$ but $\Delta^0_K(m)\neq 0 )$ then $$2 + \lfloor \ln_M p \rfloor \leq mincol_{p, m} (K) ,$$ where $M = \max \{ |m|, |m-1|  \}$.
\end{theorem}
\begin{proof}
According to Lemmas \ref{lem:MatNm} and  \ref{lem:dets}
\[
p \leq \det A_j \leq M^{N-1}
\]
Removing $\det A_j$ from the inequalities and taking logarithms base $M$ yields the result.
\end{proof}

{\bf Remark} Care should be taken when choosing the representative $m$ for the formula $2 + \lfloor \ln_M p \rfloor $. Let us consider the case of Fox colorings while keeping the notation from Theorem \ref{thm:satohminbis}. Here either $m=-1$ or $m=p-1$. Since $N$ represents the number of colors in the diagram, and we know that it takes at least $3$ colors to obtain a non-trivial coloring (Proposition \ref{prop:min3}), then the relation $p \leq (p-1)^{N-1}$ is a trivial relation (for each  $N>2$, which are the relevant $N$'s here). On the other hand, the relation $p \leq (1-(-1))^{N-1}$ is not a trivial relation.

\section{Illustrative examples.}\label{sect:examples}

In this Section we present other examples and the associated calculations. Each figure portrays a knot diagram equipped with a non-trivial $(p, m)$-coloring along with the corresponding palette graph and spanning tree. We recall that the boxed pair ``prime, integer'' at the top left of each figure, stands for the $(n, m)$ parameters of the linear Alexander quandle being used as target quandle in the figure. Later (Section \ref{sect:reducing}) we will try to reduce the number of the colors in these colorings. If the number of colors we are left with at the end  equals the lowest bound dictated by Theorem \ref{thm:satohmin} or Proposition \ref{prop:min3} or Proposition \ref{prop:conditionformin4}- we know we reached the corresponding $mincol_{p, m} (K)$, for the knot $K$ at stake.

In general, each figure intends to represent a different feature within the non-trivial $(p, m)$-colorings. For instance, Figure \ref{fig:6-1}, for knot $6_1$, with $m=2$, the determinant of the coloring matrix is $0$, which means that the diagram in the Figure is colored integrally (the coloring conditions at each crossing are satisfied over the integers). This phenomenon does not occur for knots (one component) and Fox colorings, since, for knots (one component) the Alexander polynomial evaluated at $-1$ is always an odd integer \cite{CFox}. Also, Theorem \ref{thm:satohmin} does not apply here.

In Figure \ref{fig:8-7-m-p} a non-trivial $(23, 2)$-coloring of the knot $8_7$ is found. In this case, Theorem \ref{thm:satohmin} provides an estimate for the lower bound of the number of colors: $2 + \lfloor \ln_2 23 \rfloor = 6$. Since we are using eight colors in this coloring we may state $mincol_{23, 2} (8_7) \leq 8$.

In Figure \ref{fig:6-2-101-4} there is a non-trivial $(101, 4)$-coloring of knot $6_2$. Here again Theorem \ref{thm:satohmin} provides an estimate for the lower bound of the number of colors: $2 + \lfloor \ln_4 101 \rfloor = 5$. Since we are using six colors in this coloring we may state $mincol_{101, 4} (6_2) \leq 6$.

We remark that in Figures \ref{fig:8-7-m-p} and \ref{fig:6-2-101-4} we have alternating reduced diagrams equipped with colorings associated to a prime determinant. Moreover, these colorings are such that distinct arcs receive distinct colors. In Subsection \ref{subsect:anti-KH behavior} we present examples of alternating reduced diagrams equipped with colorings associated to a prime determinant but such that different arcs receive the same color - we call this the anti KH behavior; we elaborate further below. This does not occur for Fox colorings.

\begin{figure}[!ht]
	\psfrag{0}{\huge$0$}
	\psfrag{1}{\huge$1$}
	\psfrag{2}{\huge$2$}
	\psfrag{3}{\huge$3$}
	\psfrag{4}{\huge$4$}
	\psfrag{5}{\huge$5$}
	\psfrag{a}{\huge$a$}
	\psfrag{b}{\huge$b$}
	\psfrag{2a-b}{\huge$c=2a-b$}
	\psfrag{6-1}{\huge$\mathbf{6_1}$}
	%\psfrag{p=7}{\huge$\mathbf{p=7}$}
    \psfrag{0, 2}{\huge$\mathbf{0, 2}$}
	\psfrag{m=2}{\huge$\mathbf{m=2}$}
	\centerline{\scalebox{.3}{\includegraphics{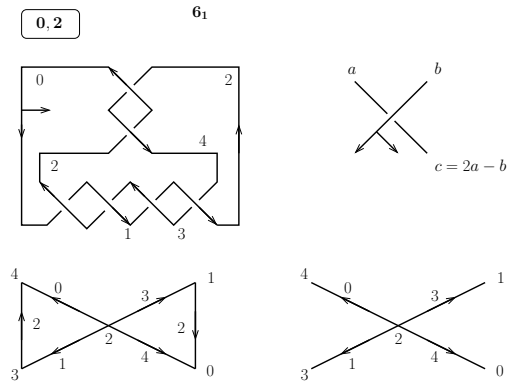}}}
	\caption{The knot $6_1$ whose $m$-determinant is $2m^2-5m+2$. For $m=2$ we obtain $2-\det = 0$. In the top row, a coloring with integers and the coloring condition. Below, the palette graph of the coloring and one of its spanning trees.}\label{fig:6-1}
\end{figure}

\begin{figure}[!ht]
%	\psfrag{8-7}{\huge$8_7$}
	\psfrag{0}{\huge$0$}
	\psfrag{1}{\huge$1$}
	\psfrag{2}{\huge$2$}
	\psfrag{3}{\huge$3$}
	\psfrag{6}{\huge$6$}
	\psfrag{12}{\huge$12$}
	\psfrag{4}{\huge$4$}
	\psfrag{19}{\huge$19$}
	\psfrag{16}{\huge$16$}
	\psfrag{21}{\huge$21$}
	\psfrag{a}{\huge$a$}
	\psfrag{b}{\huge$b$}
	\psfrag{p=23}{\huge$\mathbf{p=23}$}
	\psfrag{c=2a-b}{\huge$c=2a-b$}
	\psfrag{8-7}{\huge$\mathbf{8_7}$}
	\psfrag{23, 2}{\huge$\mathbf{23, 2}$}
	\psfrag{m=2}{\huge$\mathbf{m=2}$}
	\centerline{\scalebox{.3}{\includegraphics{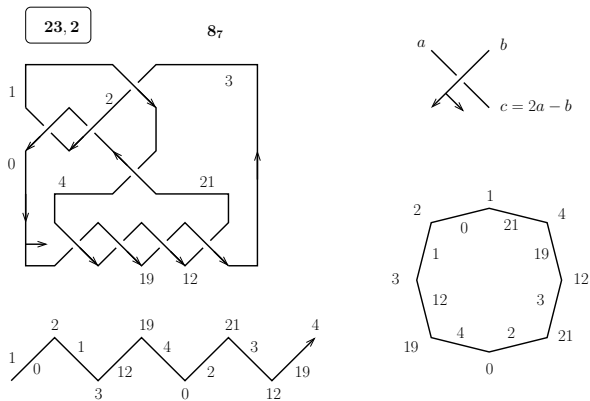}}}
	\caption{The knot $8_7$ whose $2$-determinant is $23$.   The graph depicted is the palette graph. The other graph is a spanning tree of the palette graph. We remark the KH behavior of the coloring.}\label{fig:8-7-m-p}
\end{figure}

\begin{figure}[!ht]
	\psfrag{0}{\huge$0$}
	\psfrag{1}{\huge$1$}
	\psfrag{6-2}{\huge$\mathbf{6_2}$}
	\psfrag{101, 4}{\huge$\mathbf{101, 4}$}
	\psfrag{4}{\huge$4$}
	\psfrag{26}{\huge$26$}
	\psfrag{45}{\huge$45$}
	\psfrag{60}{\huge$60$}
	\psfrag{5}{\huge$5$}
	\psfrag{a}{\huge$a$}
	\psfrag{b}{\huge$b$}
	\psfrag{c=4a-3b}{\huge$c=4a-3b$}
	\psfrag{p=19}{\huge$\quad\mathbf{p=19}$}
	\centerline{\scalebox{.3}{\includegraphics{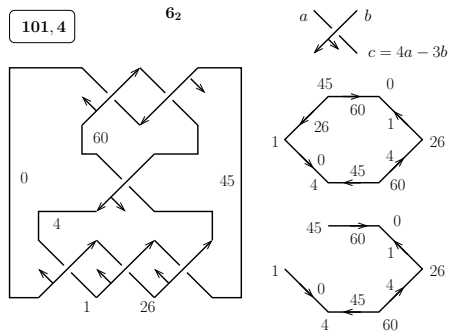}}}
	\caption{Non-trivial colorings for $6_2$; $m=4$ and $p=101$.}\label{fig:6-2-101-4}
\end{figure}

\subsection{Further examples: anti-KH behavior.}\label{subsect:anti-KH behavior}

A conjecture in \cite{HK} has subsequently become a Theorem, proven by Mattman and Solis \cite{msolis}. Here is its statement.

\begin{theorem}
Let $K$ be an alternating knot $($one component$)$ of prime determinant, $p$, i.e., $\Delta^0_K(-1) = p$. Any non-trivial $p$-coloring $($Fox coloring$)$ on a reduced alternating diagram of $K$ assigns different colors to different arcs.
\end{theorem}

\begin{definition}
We say a knot, $K$, displays KH behavior mod $(p, m)$, if $K$  is alternating, and if there is a prime $p$ and an integer $m$ such that $1<m<p$ with $\Delta^0_K(m) = p$, and for some reduced alternating diagram of $K$ equipped with a non-trivial $(p, m)$-coloring, different arcs receive different colors. Otherwise we say $K$  displays anti-KH behavior.
\end{definition}

Figures \ref{fig:6-1-m-3}, \ref{fig:7-2new}, and \ref{fig:9-12-3-11} present cases of anti-KH behavior, namely reduced alternating diagrams, prime determinant of the associated coloring matrices, but the non-trivial colorings on these diagrams assign the same colors to some distinct arcs.

\begin{figure}[!ht]
	\psfrag{0}{\huge$0$}
	\psfrag{1}{\huge$1$}
	\psfrag{2}{\huge$2$}
	\psfrag{3}{\huge$3$}
	\psfrag{4}{\huge$4$}
	\psfrag{5}{\huge$5$}
	\psfrag{a}{\huge$a$}
	\psfrag{b}{\huge$b$}
	\psfrag{3a-2b}{\huge$c=3a-2b$}
	\psfrag{6-1}{\huge$\mathbf{6_1}$}
	%\psfrag{p=7}{\huge$\mathbf{p=7}$}
	\psfrag{5, 3}{\huge$\mathbf{5, 3}$}
	\centerline{\scalebox{.35}{\includegraphics{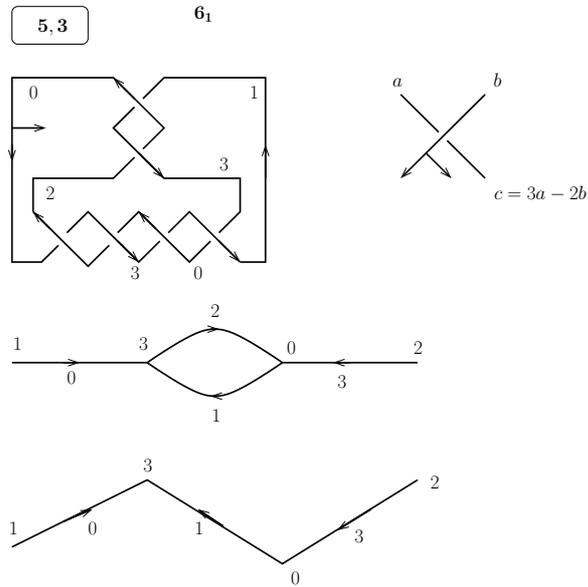}}}
	\caption{The knot $6_1$ whose $m$-determinant is $2m^2-5m+2$. For $m=3$ we obtain $3-\det = 5$. In the top row, a coloring with integers and the coloring condition. Below, the palette graph of the coloring and one of its spanning trees.}\label{fig:6-1-m-3}
\end{figure}

\begin{figure}[!ht]
	\psfrag{0}{\huge$0$}
	\psfrag{1}{\huge$1$}
	\psfrag{2}{\huge$2$}
	\psfrag{3}{\huge$3$}
	\psfrag{4}{\huge$4$}
	\psfrag{5}{\huge$5$}
	\psfrag{a}{\huge$a$}
	\psfrag{b}{\huge$b$}
	\psfrag{2a-b}{\huge$c=2a-b$}
	\psfrag{7-2}{\huge$\mathbf{7_2}$}
	\psfrag{5, 2}{\huge$\mathbf{5, 2}$}
	\psfrag{m=2}{\huge$\mathbf{m=2}$}
	\centerline{\scalebox{.35}{\includegraphics{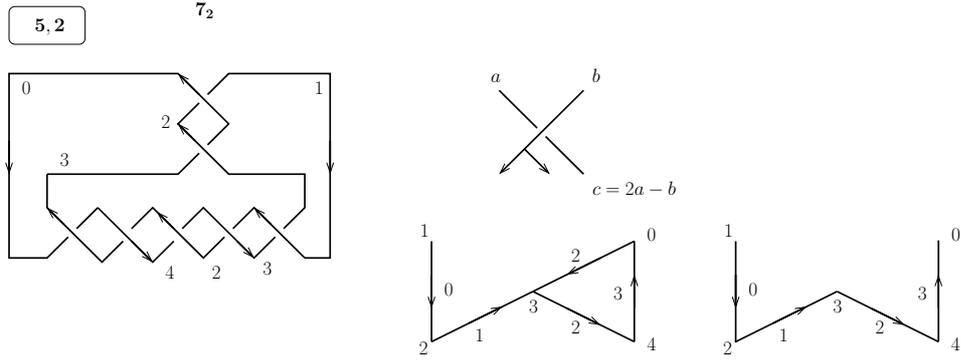}}}
	\caption{The knot $7_2$ whose $m$-determinant is $3m^2-5m+3$. For $m=2$ we obtain $2-\det = 5$. We remark the anti-KH behavior of this coloring.}\label{fig:7-2new}
\end{figure}

\begin{figure}[!ht]
	\psfrag{0}{\huge$0$}
	\psfrag{1}{\huge$1$}
	\psfrag{2}{\huge$2$}
	\psfrag{3}{\huge$3$}
	\psfrag{4}{\huge$4$}
	\psfrag{5}{\huge$5$}
	\psfrag{6}{\huge$6$}
	\psfrag{10}{\huge$10$}
	\psfrag{9}{\huge$9$}
	\psfrag{a}{\huge$a$}
	\psfrag{b}{\huge$b$}
	\psfrag{c}{\huge$c=3a-2b$}
	\psfrag{9-12}{\huge$\mathbf{9_{12}}$}
	\psfrag{3-11}{\huge$3; 11$}
	\psfrag{6-3}{\huge$\mathbf{6_3}$}
	\psfrag{11, 3}{\huge$\mathbf{11, 3}$}
	\psfrag{m=2}{\huge$\mathbf{m=2}$}
	\centerline{\scalebox{.35}{\includegraphics{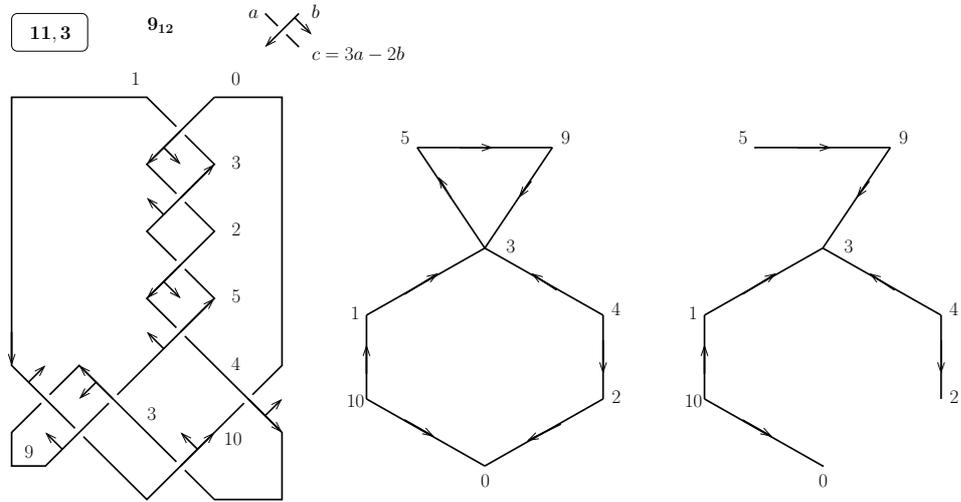}}}
	\caption{The knot $9_12$ whose $m$-determinant is $2m^4 - 9m^3 + 13m^2 - 9m + 2$. For $m=3$ we obtain $3-\det = 11$.  We note the anti-KH behavior: minimal alternating diagram, prime determinant but repeated colors (3).}\label{fig:9-12-3-11}
\end{figure}

\section{Equivalence classes of colorings by linear Alexander quandles}\label{sect:autocolorings}

We recall here the definition of Linear Alexander Quandle also to introduce notation which will shorten the statements in the sequel.

\begin{definition}[Linear Alexander Quandles]\label{def:LAQ}
Let $p$ be an odd prime and $m$ a positive integer such that $1<m<p$. We let $\mathbf{LAQ(p, m)}$ stand for the linear Alexander quandle of order $p$ and parameter $m$ i.e., the quandle whose underlying set is the integers modulo $p$, $\mathbf{Z}_p$, and whose quandle operation is, for any $a, b\in \mathbf{Z}_p$, $$a\ast b := ma + (1-m)b \quad \text{ mod } p.$$ We remark that for $m=p-1$ we obtain the dihedral quandle of order $p$.
\end{definition}

\begin{definition}[Automorphism groups of Linear Alexander Quandles]\label{def:Aut(LAQ)}
Let $p$ be an odd prime and $m$ a positive integer such that $1<m<p$. We let $\mathbf{Aut(p, m)}$ stand for the group of automorphisms of the linear Alexander quandle of order $p$ and parameter $m$, $\mathbf{LAQ(p, m)}$. This is the set of bijections $f : \mathbf{Z}_p \longrightarrow \mathbf{Z}_p$ such that, for any $a, b\in \mathbf{Z}_p$, $f(ma+(1-m)b) = mf(a)+(1-m)f(b)$; the group operation is composition of functions.
\end{definition}

\begin{theorem}[\cite{Nelson, XdHou, elhamdadi}]\label{thm:auto}
Let $p$ be an odd prime and $m$ a positive integer such that $1<m<p$. Then $$\mathbf{Aut(p, m)} \cong \mathbf{Aff(p)} \cong \mathbf{Z}_p \rtimes \mathbf{Z}_p^\times ,$$ where $\mathbf{Aff(p)} \big( \cong \mathbf{Z}_p \rtimes \mathbf{Z}_p^\times \big) $ is the affine group over $\mathbf{Z}_p$, \cite{Rotman}. In particular, $\mathbf{Aut(p, m)}$ only depends on $p$.
\end{theorem}
\begin{proof}
This proof is contained in the proof of Theorem $2.1$ in \cite{Nelson}, see also \cite{XdHou}, and \cite{elhamdadi} for the case $m=p-1$. Here we provide a direct calculation in the spirit of \cite{elhamdadi}.

We start by proving that $\mathbf{Aff(p)} \subset \mathbf{Aut(p, m)} $. For $\lambda \in \mathbf{Z}_p^\times$ and $\mu\in \mathbf{Z}_p$, consider $$ f_{\lambda,\mu} (x) = \lambda x + \mu \qquad \text{ for each } \quad x\in \mathbf{Z}_p.$$ For any $a, b \in \mathbf{Z}_p$,
\begin{align*}
f_{\lambda, \mu}(ma+(1-m)b) & = \lambda (ma+(1-m)b) + \mu = m(\lambda a + \mu ) + (1-m)(\lambda b + \mu ) = \\
& =  m f_{\lambda, \mu}(a) + (1-m)f_{\lambda, \mu}(b) .
\end{align*}

We now prove $\mathbf{Aut(p, m)} \subset \mathbf{Aff(p)}$. We pick $f\in \mathbf{Aut(p, m)}$ i.e. such that for any $a, b\in \mathbf{Z}_p$, $$f(ma+(1-m)b) = mf(a)+(1-m)f(b) .$$ Then, for any $x\in \mathbf{Z}_p$, set $$g(x) = f(x) - f(0) .$$ Then $g(0)=f(0)-f(0)=0$. Moreover,
\begin{align*}
&g(ma+(1-m)b) = f(ma+(1-m)b) - f(0) =  mf(a)+(1-m)f(b) - f(0) = \\
&= m(f(a)-f(0))+(1-m)(f(b) - f(0)) = mg(a)+(1-m)g(b)
\end{align*}
In particular, setting $b=0$, we obtain $$g(ma) = mg(a) ;$$ whereas setting $a=0$, we obtain $$g((1-m)b) = (1-m)g(b).$$ It follows that for any positive integer $k$, $g(m^ka) = m^kg(a)$ and $g((1-m)^kb) = (1-m)^kg(b)$. In particular, there are positive integers $k_1, k_2$ such that $m^{k_1}=_p 1 =_p (1-m)^{k_2}$. Thus, for any $x\in \mathbf{Z}_p$, $$ g(m^{-1}x) =_p g(m^{k_1-1}x) =_p m^{k_1-1}g(x) =_p m^{-1}g(x)$$ $$g((1-m)^{-1}x) =_p g((1-m)^{k_2-1}x) =_p (1-m)^{k_2-1}g(x) =_p (1-m)^{-1}g(x) .$$
\begin{cl}\label{cl:1}
Let $x\in \mathbf{Z}_p$. For any  $k\in \mathbf{Z}$, $$g(mkx)=mkg(x).$$
\end{cl}
\begin{proof}

The proof is clear for $k=0$. We now prove for  positive  $k$, by induction.

For $k=1$ we already know $g(mx)=mg(x)$. Now for the inductive step.
\begin{align*}
&g(m(k+1)x) = g(mkx+mx) = g(mkx+(1-m)(1-m)^{-1}mx) = mg(kx)+(1-m)g((1-m)^{-1}mx) = \\
&= g(mkx)+(1-m)(1-m)^{-1}g(mx) = mkg(x)+mg(x) = m(k+1)g(x)
\end{align*}
The proof is complete for non-negative $k$. For negative $k$ we write $kx = (-k)(-x)$ and invoke the first part of the proof. The proof is complete.
\end{proof}

\begin{cl}
Let $x\in \mathbf{Z}_p$, let $k, l\in \mathbf{Z}$.  Then, $$g((mk+l)x) = (mk+l)g(x) .$$
\end{cl}
\begin{proof}
 \begin{align*}
&g((mk+l)x) = g(m(kx)+lx) = g(m(kx)+(1-m)((1-m)^{-1}lx)) = \\
&= mg(kx)+(1-m)g((1-m)^{-1}(lx)) = mg(kx)+(1-m)(1-m)^{-1}g(lx) = \\
&= mkg(x)+lg(x) = (mk+l)g(x) .
\end{align*}
The proof is complete.
\end{proof}

So for any $\lambda_0\in \mathbf{Z}_p^\times$, pick positive integers $k, l$ such that $\lambda_0 = mk+l$ and $0\leq l < m$. Then $$g(\lambda_0 x) = g((mk+l)x) = (mk+l)g(x) = \lambda_0 g(x)$$ so that $g$ is linear i.e., there is $\lambda\in \mathbf{Z}_p^\times$ such that $g(x)=\lambda x$, for any $x\in \mathbf{Z}_p$. Thus $$f(x)=g(x)+f(0)=\lambda x + f(0) = \lambda x + \mu \qquad \text{ for some } \quad \mu \in \mathbf{Z}_p$$ which completes the proof of Theorem \ref{thm:auto}.

\end{proof}

\begin{theorem}
The equivalence classes of colorings are preserved by colored Reidemeister moves. This means the following. For some odd prime $p$ and integer $m$ such that $1<m<p$, consider a knot/link diagram admitting nontrivial  $(p, m)$-colorings  and equipped with such a non-trivial coloring, call it ${\cal C}$. Assume a Reideimeister move is performed on this colored diagram and the colors of the resulting diagram,  outside the neighborhood where the transformation took place, remain the same. The colors inside this neighborhood are changed in a unique way so that a non-trivial $(p, m)$-coloring is obtained in the resulting diagram, call it ${\cal C'}$. $($This is a colored Reidemeister move, see Figure \ref{fig:Reidemeister+quandle}, and \cite{pLopes}.$)$ Now suppose $f\in \mathbf{Aut(p, m)}$ and consider the new coloring on the original diagram given by $f({\cal C})$. Then the same colored Reidemeister move applied to  $f({\cal C})$ will give rise to $f({\cal C'})$.
\end{theorem}
\begin{proof}
The proof is an easy verification using the quandle properties. We illustrate the procedure by showing how it works for the type II Reidemeister move in Figure \ref{fig:r2}. Check \cite{GJKL} for the case of dihedral quandles.
\begin{figure}[!ht]
	\psfrag{D}{\huge$D$}
	\psfrag{D'}{\huge$D'$}
	\psfrag{calCr=x}{\huge${\cal C}(r)=x$}
	\psfrag{calCs=y}{\huge${\cal C}(s)=y$}
	\psfrag{fcalCr=x}{\huge$g\big( {\cal C}(r)\big)=g(x)$}
	\psfrag{fcalCs=y}{\huge$g\big( {\cal C}(s)\big)=g(y)$}
	\psfrag{calC'r1=x}{\huge${\cal C'}(r'_1)=x \qquad {\cal C'}(s')=y$}
	\psfrag{calC'r2=2y-x}{\huge${\cal C'}(r'_2)=  mx+(1-m)y$}
	\psfrag{calC'r3=x}{\huge${\cal C'}(r'_3)=  m^{-1}(mx+(1-m)y)+(1-m^{-1})y = x$}
	\psfrag{fcalC'r1=x}{\huge$g\big( {\cal C'}(r'_1)\big)=g(x) \qquad g\big( {\cal C'}(s')\big)=g(y)$}
	\psfrag{fcalC'r2=2y-x}{\huge${g\big( \cal C'}(r'_2)\big)= mg(x)+(1-m)g(y)$}
	\psfrag{fcalC'r3=x}{\huge${g\big( \cal C'}(r'_3)\big)= g(m^{-1}(mx+(1-m)y)+(1-m^{-1})y) = g(x)$}
	\psfrag{r}{\huge$r$}
	\psfrag{s}{\huge$s$}
	\psfrag{s'}{\huge$s'$}
	\psfrag{r1}{\huge$r'_1$}
	\psfrag{r2}{\huge$r'_2$}
	\psfrag{r3}{\huge$r'_3$}
	\psfrag{tilde}{\huge$\sim$}
	\psfrag{9}{\huge$\mathbf{9}$}
	\centerline{\scalebox{.35}{\includegraphics{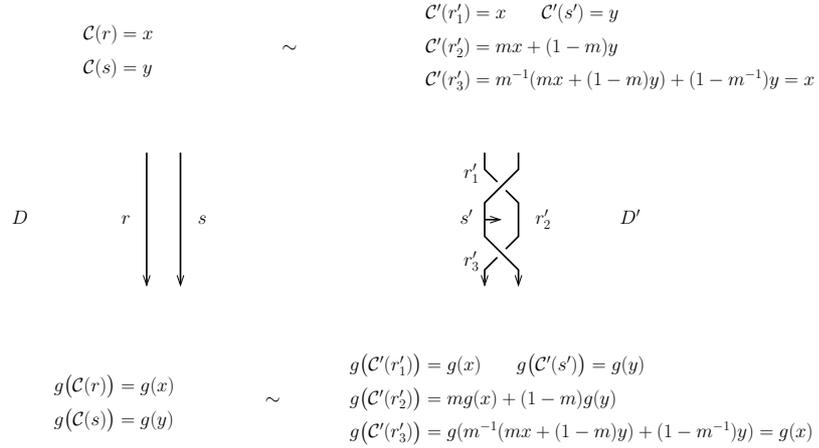}}}
	\caption{Behavior of equivalence classes under type $II$ Reidemeister move for a diagram endowed with a non-trivial $(p, m)$-coloring. $D$ and $D'$ are two knot diagrams related by a type $II$ Reidemeister move as depicted. ${\cal C}$ stands for a non-trivial $(p, m)$-coloring on $D$ i.e., a map from the set of arcs of the diagram into $\mathbf{Z}/p\mathbf{Z}$, satisfying the $(p, m)$-coloring condition at each crossing of the diagram (the same for ${\cal C'}$ with respect to $D'$). $g\in Aut_{(p, m)}$. $\tilde{ }$ means the colorings are related. The Figure shows that, under a type $II$ Reidemeister move, related colorings are taken to related colorings.}\label{fig:r2}
\end{figure}
\end{proof}

\begin{theorem}\label{thm:numberequivclasses}
Assume that a non-split link $K$ $($possibly a knot$)$ is such that there is an integer $m>1$ such that $\Delta^0_K (m) =p$, for some odd prime $p$. Then any diagram of $K$ has only one equivalence class of $(p, m)$-colorings.
\end{theorem}
\begin{proof}
Consider the knot or non-split link, $K$, whose reduced Alexander polynomial is denoted $\Delta^0_K$.
Assume further that there exists a positive integer $m$ such that $\Delta^0_K (m) = p$, for some odd prime, $p$.
This means there is a certain reduced coloring matrix (over the integers) whose determinant is $\pm  m^k \cdot p$. Then its Smith Normal Form has a number of $1$'s and exactly one $m^k\cdot p$ along the diagonal. This means that the number of $(p, m)$-colorings for $K$ is $p\cdot p$ where the other $p$ factor comes from the fact that the coloring matrix has zero determinant. There are then $p^2-p$ non-trivial $(p, m)$-colorings for $K$ for there is one trivial $(p, m)$-coloring per each of the $p$ colors available.

Now the automorphism group of $(p, m)$-colorings is the semi-direct product $\mathbf{Z}_p \rtimes \mathbf{Z}_p^{\ast}$, whose action upon the set of non-trivial $(p, m)$-colorings over any given diagram is free (if two group elements move a coloring the same way then they are the same group element). The proof follows. If there are two group elements $f$ and $g$ such that, for any non-trivial $(p, m)$-coloring, $\cal C$, $f({\cal C}) = g({\cal C})$, then there are colors $a$ and $b$ such that $f(a)=g(a), f(b)=g(b)$ and so for certain $\lambda, \lambda' \in \mathbf{Z}_p^\ast$ and $\mu, \mu'\in \mathbf{Z}$, we have $\lambda a + \mu = b, \quad \lambda' a + \mu' = b$ which yields $\lambda = \lambda', \mu = \mu'$, since $\mathbf{Z}_p$ is a field.

Thus the number of elements in any orbit of this action is the number of elements of the group which is $p^2-p$. But this is the number of non-trivial $(p, m)$-colorings. There is then only one orbit for prime p.

\end{proof}

\begin{cor}
In the conditions of Theorem \ref{thm:numberequivclasses}, if $D$ is a diagram of $K$ such that there exist one non-trivial $(p, m)$-coloring $D$ using the minimum number of distinct colors, then all other non-trivial colorings of this diagram use the minimum number of colors.
\end{cor}
\begin{proof}
Since the group acts by automorphisms, which are bijections, any two colorings of the same orbit have the same number of distinct colors. Since there is only one orbit, the proof is complete. (In fact, we proved something stronger: for any diagram of such a $K$, any two non-trivial $(p, m)$-colorings on this diagram, use the same number of colors.)
\end{proof}

\section{Further obstructions to minimization.}\label{sect:obstructions}

In this Section, we prove there are sets of colors that can never constitute the set of colors of a non-trivial $(p, m)$-coloring. These results are important when trying to reduce the number of colors in  such a coloring: they indicate which colors must not be removed.

\begin{theorem}\label{thm:obstructions}
Let $K$ be a knot, $p$ an odd prime and $m$ an integer such that $1<m<p$, along with $$\Delta^0_{K} (m) = p .$$
 Then, $K$ admits non-trivial $(p, m)$-colorings.

 Let $S$ be a subset of representatives of  $\mathbf{Z}_p$, defined as follows:
\begin{enumerate}
\item If $1 < m < p/2$, then  $$S = \{ a\in \mathbf{Z}\, |\, 0\leq a < p/m \};$$
\item if $(p+1)/2 \leq m < p$, then $$S = \{ a\in \mathbf{Z}\, |\, 0\leq a < p/(p+1-m) \}.$$
\end{enumerate}
Then, a non-trivial $(p, m)$-coloring of $K$ cannot be obtained with colors from $S$ alone.
\end{theorem}
\begin{proof}
We keep the notation and terminology from the statement.
\begin{enumerate}
\item If $1 < m < p/2$, then a $(p, m)$-coloring condition is of the form:
\[
c = ma+(1-m)b \Longleftrightarrow ma = (m-1)b + c
\]
Assume now $a, b, c \in S$. Then $$0 \leq ma < m \cdot p/m = p \quad \text{ and } \quad 0\leq (m-1)b + c < (m-1) \cdot p/m + p/m = p, $$ which means the expression $c = ma+(1-m)b$ holds over $\mathbf{Z}$ which further implies that $\Delta^0_{K} (m)(=p)$ has to be divisible by infinitely many primes. This is absurd. The proof is complete for the case $1 < m \leq p/2$.
\item If $(p+1)/2 \leq m < p$, then an $(p, m)$-coloring condition is of the form:
\[
c = ma+(1-m)b \Longleftrightarrow [p - (m-1)]b = [p - m]a + c
\]
An argument analogous to the one in the first item now follows. If $a, b, c\in S$ then, $$0 \leq [p - (m-1)]b <   [p - (m-1)]\cdot p/[p + 1  - m] = p, $$ and $$ 0\leq [p - m]a + c <   [p - m]\cdot p/[p - (m-1)] + p/[p - (m-1)] = p, $$ and the same conclusion as in the first item follows. The proof is complete.
\end{enumerate}
\end{proof}

\begin{cor}
We keep the notation and terminology from Theorem \ref{thm:obstructions}. Assume further $f\in Aut (p, m)$. Then, the colors from the set $f(S) = \{ f(a) \, |\, a\in S  \}$ cannot give rise to a non-trivial $(p, m)$-coloring of $K$.
\end{cor}
\begin{proof}
Let $g$ be the inverse of $f$. If $f(S)$ could give rise to a non-trivial $(p, m)$-coloring of $K$ then, $g(f(S)) = S$ could also give rise to a non-trivial $(p, m)$-coloring of $K$. But this contradicts Theorem \ref{thm:obstructions}. This completes the proof.
\end{proof}

\begin{definition}\label{def:suffset}
Let $K$ be a knot admitting non-trivial $(p, m)$-colorings where $p$ is prime and $m$ is an integer such that $1<m<p$ and $\Delta^0_K(m)=p$. A subset $S$ of $\mathbf{Z}_p$ is said to be a {\bf{$\mathbf{(p, m)}$-Sufficient Set of Colors for $\mathbf{\text{K}}$}} if there is a diagram of $K$ which supports a non-trivial $(p, m)$-coloring, using colors from $S$, alone.
\end{definition}

\begin{cor}\label{cor:kauffmansaito}
We keep the notation and terminology from Theorem \ref{thm:obstructions}.

Suppose $S= \{   c_1, \dots , c_n   \}$ is a $(p, m)$-Sufficient Set of Colors for $K$. $($Without loss of generality we take the $c_i$'s from $\{ 0, 1, 2, \dots \, , p-1 \}$.$)$ For each $c_i$ $(i=1, 2, \dots , n)$, let $S_i$ be the set of ordered pairs $( c_i^1, c_i^2 )$ $($with $c_i^1\neq c_i^2 \text{ from } S$$)$ such that $mc_i^1 + (1-m)c_i = c_i^2$, mod $p$.

Then,\begin{enumerate}
\item There is at least one $i$ such that the expression $mc_i^1 + (1-m)c_i = c_i^2$ does not make sense over the integers, it only makes sense modulo $p$.

\item If there is $i\in \{ 1, 2, \dots , n \}$ such that, for all $j\in \{ 1, 2, \dots , n  \}$, $c_i$ does not belong to any of the ordered pairs in $S_j$, then $n> mincol_p\, L$ and  $S\setminus \{ c_i \}$ is also a $(p, m)$-Sufficient Set of Colors.
\end{enumerate}
\end{cor}
\begin{proof}
\begin{enumerate}
\item If for all $i$'s each of the  expressions $mc_i^1 + (1-m)c_i = c_i^2$ makes sense without considering it mod $p$, then we have colorings whose coloring conditions hold modulo any prime. But this conflicts with Theorem \ref{thm:obstructions}.
\item Since $c_i$ cannot be the color of an under-arc at a polychromatic crossing, then this implies one of two possibilities. Either $c_i$ is not being used at all, and the proof is complete. Or $K$ is a split link (see Definition \ref{def:splitlink} and Proposition \ref{prop:Alexanderlink}, right below) such that one or more link components are colored with $c_i$
alone. But in this case, $\Delta^0_K(m)=0$ over the integers, which conflicts with the hypothesis.
\end{enumerate}
This completes the proof.
\end{proof}

The following Proposition is a well-known fact proved by Alexander \cite{Alexander}. We state and prove it here for the sake of completeness since we use it in the proof of Corollary \ref{cor:kauffmansaito}.

\begin{prop}\label{prop:Alexanderlink}
The Alexander polynomial of a split link, $K$, is identically $0$.
\end{prop}
\begin{proof}
Consider $K$ after the deformation which places some of its components in one ball and the others in the other ball, the two balls being disjoint, call them $B_1$ and $B_2$. Then we can organize the Alexander coloring conditions and corresponding variables for each of the balls independently. Then the (Alexander) coloring matrix is a block-diagonal matrix, one block per ball. But each of these block diagonal matrices is such that upon addition of all its columns we obtain a column of $0$'s. Then the first minor determinant (i.e., the Alexander polynomial of $K$) of the over-all matrix has to be $0$. This completes the proof.
\end{proof}

\section{Minimizing by direct calculation.}\label{sect:minimizing}

We now state and prove a result which will give a sharper estimate for the lower bounds for the minimum number of colors for non-trivial $(p, m)$-colorings, in some particular cases.

\begin{prop}\label{prop:conditionformin4}
Suppose $K$ is a knot whose Alexander polynomial is not identically $0$, and that $p$ is an odd prime and $m$ an integer such that $1<m<p$. Assume further that $K$ admits non-trivial $(p, m)$-colorings.

Such a non-trivial coloring over a diagram of $K$ needs at least $4$ different colors, if none of the conditions in $(\ast)$ below holds.$$  (\ast)  \qquad  \qquad  m=_p 2  \qquad \qquad m^2 - m +1 =_p 0  \qquad \qquad  2m-1=_p 0 \qquad  .$$
\end{prop}
\begin{proof}
Any non-trivial $(p, m)$-coloring with $p$ and $m$ as in the statement, is such that three distinct colors, say $a, b, c$, have to meet at a crossing as depicted in the top left of Figure \ref{fig:m-xing2}, see statement and proof of Proposition \ref{prop:min3}.
\begin{figure}[!ht]
	\psfrag{a}{\huge$a$}
	\psfrag{b}{\huge$b$}
	\psfrag{c}{\huge$c$}
	\psfrag{c = ma+(1-m)b}{\huge$c$}
	\psfrag{(i)}{\huge$(i)$}
	\psfrag{(ii)}{\huge$(ii)$}
	\psfrag{(iii)}{\huge$(iii)$}
	\psfrag{(iv)}{\huge$(iv)$}
	\psfrag{a(c)}{\huge$a(c)$}
	\psfrag{c(a)}{\huge$c(a)$}
	\centerline{\scalebox{.35}{\includegraphics{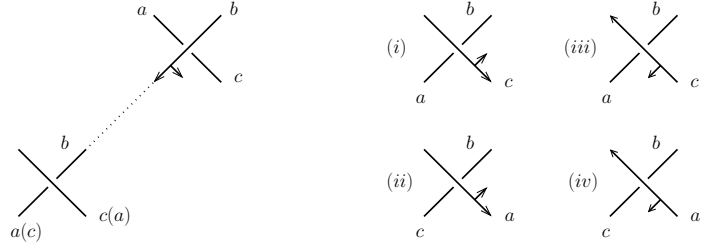}}}
	\caption{The setting with only three colors. In the lower left, the crossing where the color $b$ ends up with the possibilities illustrated for this crossing. These possibilities materialize into $(i), (ii), (iii), \text{ and }(iv)$ on the right hand-side.}\label{fig:m-xing2}
\end{figure}
Since the color $b$ stemming from the over-arc of the crossing in the top left of Figure \ref{fig:m-xing2} has to end up at a polychromatic crossing (otherwise the link would be split) as depicted in the bottom left of the same Figure, the condition $c=ma +(1-m)b$ has to comply with one of the coloring conditions associated with $(i), (ii), (iii), \text{ or }(iv)$. We  treat them now.
\begin{enumerate}
\item $(i)$:
\[
\begin{cases}
c=ma+(1-m)b\\
b=ma+(1-m)c
\end{cases}
\]Subtracting the bottom expression from the first and simplifying yields $(m-2)(b-c)=0$. Since $b$ and $c$ are distinct then $m=2$ mod $p$.

\item $(ii)$:
\[
\begin{cases}
c=ma+(1-m)b\\
b=mc+(1-m)a
\end{cases}
\]Substituting $b$ in the top equation for its expression in the bottom equation yields $(m^2-m+1)(a-c)=0$. Since $a$ and $c$ are distinct then $m^2-m+1=0$ mod $p$.

\item $(iii)$:
\[
\begin{cases}
c=ma+(1-m)b\\
a=mb+(1-m)c
\end{cases}
\]Substituting $a$ in the top equation for its expression in the bottom equation yields $(m^2-m+1)(b-c)=0$. Since $b$ and $c$ are distinct then $m^2-m+1=0$ mod $p$.

\item $(iv)$:
\[
\begin{cases}
c=ma+(1-m)b\\
c=mb+(1-m)a
\end{cases}
\]Simplifying yields $(2m-1)(b-a)=0$. Since $b$ and $a$ are distinct then $2m-1=0$ mod $p$.
\end{enumerate}
Therefore, if $a, b, c$ are distinct, mod $p$, and one of the ($\ast$) conditions is satisfied, then a fourth color is needed to assemble such a non-trivial coloring. the proof is concluded.
 \end{proof}

We remark that $m=-1$ mod $3$ for $p=3$ complies with each expression in $(\ast)$.

Note that the non-trivial Fox $3$-coloring of the trefoil (left-hand side of Figure \ref{fig:tri-coloring-trefoil-unknot}) is a non-trivial $(3, 2)$-coloring. Here $mincol_{3, 2}(Trefoil) = 3$. This does not conflict with Proposition \ref{prop:conditionformin4} above since the reduced Alexander polynomial for the trefoil is $\Delta^0_{Trefoil} (m) = m^2-m+1$, so that $\Delta^0_{Trefoil} (2) = 3 =_{3} 0$ (also for $(7, 3)$, $mincol_{7, 3}(Trefoil) = 3$).

\begin{cor}\label{cor:minbyhand}
Let $m$ and $p$ be as in Proposition \ref{prop:conditionformin4}. Furthermore, if $m>2$ and $p>m^2-m+1\,   ($over $\mathbf{Z})$, then none of the conditions in $(\ast)$ is satisfied. Thus any  non-split link admitting a non-trivial $(p, m)$-coloring for such $p, m$ needs at least $4$ distinct colors.
\end{cor}
\begin{proof}
Noting that $$m^2-m+1 - (2m-1) = m^2-3m+2 = (m-2)(m-1),$$ it follows that $$2 < m < 2m-1 < m^2-m+1 < p, $$ which completes the proof.
\end{proof}
\begin{cor}\label{cor:minbyhand2}
Let $m$ and $p$ be as in Proposition \ref{prop:conditionformin4} and $m>2$ and $p>m^2-m+1\,   ($over $\mathbf{Z})$. Assume further that $m^2 >p$. Then  any  non-split link admitting a non-trivial $(p, m)$-coloring for such $p, m$ needs at least $4$ distinct colors which is a sharper estimate than the one given by Theorem \ref{thm:satohmin}.
\end{cor}
\begin{proof}
If $m^2>p$ then $2 + \lfloor \ln_m p \rfloor \leq 3 < 4$, which completes the proof.
\end{proof}
In Figure \ref{fig:6-2-LAQ} we show an example illustrating Corollary \ref{cor:minbyhand2}.
\begin{figure}[!ht]
	\psfrag{0}{\huge$0$}
	\psfrag{1}{\huge$1$}
	\psfrag{6-2}{\huge$\mathbf{6_2}$}
	\psfrag{19, 3}{\huge$\mathbf{19, 3}$}
	\psfrag{2}{\huge$2$}
	\psfrag{3}{\huge$3$}
	\psfrag{7}{\huge$7$}
	\psfrag{17}{\huge$17$}
	\psfrag{5}{\huge$5$}
	\psfrag{a}{\huge$a$}
	\psfrag{b}{\huge$b$}
	\psfrag{c=3a-2b}{\huge$c=3a-2b$}
	\psfrag{p=19}{\huge$\quad\mathbf{p=19}$}
	\centerline{\scalebox{.35}{\includegraphics{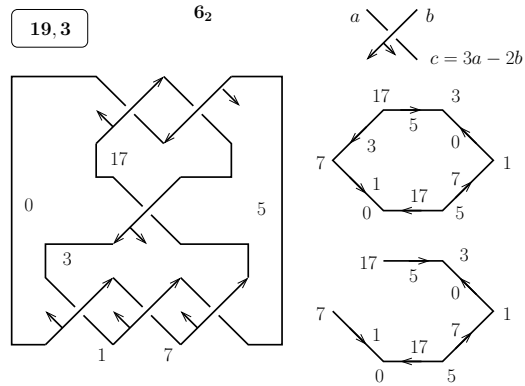}}}
	\caption{Non-trivial colorings for $6_2$; $m=3$ and $p=19$.}\label{fig:6-2-LAQ}
\end{figure}

We will look into some the cases not covered by Corollary \ref{cor:minbyhand}: $m=2$ or $2m-1=0$ or $p=m^2-m+1$. As already mentioned, one such particular case is $p=3$
and $m=-1$, and the trefoil admits such non-trivial colorings which involve exactly three colors - over any diagram of the trefoil.

Table \ref{Ta:Behavior2m-1-m2-m+1} shows the behavior of quantities $2m-1$ and $m^2-m+1$ as a function of $m$ for the first few positive integers, for the reader's convenience.

\begin{table}[h!]
\begin{center}
\begin{tabular}{| c ||  c |   c  |     c  |   c  |   c  |   c  |   c  |   c  |   c  |   c  |   c  |   c  |   }\hline
    $m$      &  $2$ &   $3$  & $4$ & $5$  & $6$ & $7$  & $8$ & $9$ & $10$ & $11$ & $12$ & $13$  \\ \hline \hline
$2m-1$ &  $3$        & $5$    &   $7$  &  $9$   &   $11$  &  $13$   &   $15$  &  $17$  &   $19$  & $21$  &   $23$  &  $25$  \\ \hline \hline
$m^2-m+1$   & $3$     & $7$   & $13$  & $3\cdot 7$  & $31$    & $43$ & $3\cdot 19$  & $73$ & $7\cdot 13$  & $3\cdot 37$  & $7\cdot 19$   & $157$    \\ \hline
\end{tabular}
\caption{Behavior of quantities $2m-1$ and $m^2-m+1$ as a function of $m$ for the first few positive integers.}
\label{Ta:Behavior2m-1-m2-m+1}
\end{center}
\end{table}

\section{Reducing the number of colors.}\label{sect:reducing}

In this Section we apply the results above to reduce as much as we are able, the number of colors in a few cases of non-trivial $(p, m)$-colorings. In Figure \ref{fig:6-1-m-3-reducing} we  obtain $mincol_{5, 3} (6_1) = 3$ noting Proposition \ref{prop:conditionformin4} cannot be applied here (since $2\cdot 3 - 1 = 5$) and that this situation is then governed by Proposition \ref{prop:min3}.

As for Figure \ref{fig:6-3-reducing}, we have a non-trivial $(7, 2)$-coloring of $6_3$ so Theorem \ref{thm:satohmin}  estimates a lower bound for $mincol_{7, 2}(6_3)$ of $2+\lfloor \ln_2 7\rfloor = 4$. Since we are only able to reduce the number of colors from $6$ (top diagram) to $5$ (bottom diagram), we can only state $mincol_{7, 2}(6_3)\leq 5$.

\begin{figure}[!ht]
	\psfrag{0}{\huge$0$}
	\psfrag{1}{\huge$1$}
	\psfrag{2}{\huge$2$}
	\psfrag{3}{\huge$3$}
	\psfrag{4}{\huge$4$}
	\psfrag{5}{\huge$5$}
	\psfrag{a}{\huge$a$}
	\psfrag{b}{\huge$b$}
	\psfrag{3a-2b}{\huge$c=3a-2b$}
	\psfrag{6-1}{\huge$\mathbf{6_1}$}
	%\psfrag{p=7}{\huge$\mathbf{p=7}$}
	\psfrag{5, 3}{\huge$\mathbf{5, 3}$}
	\centerline{\scalebox{.35}{\includegraphics{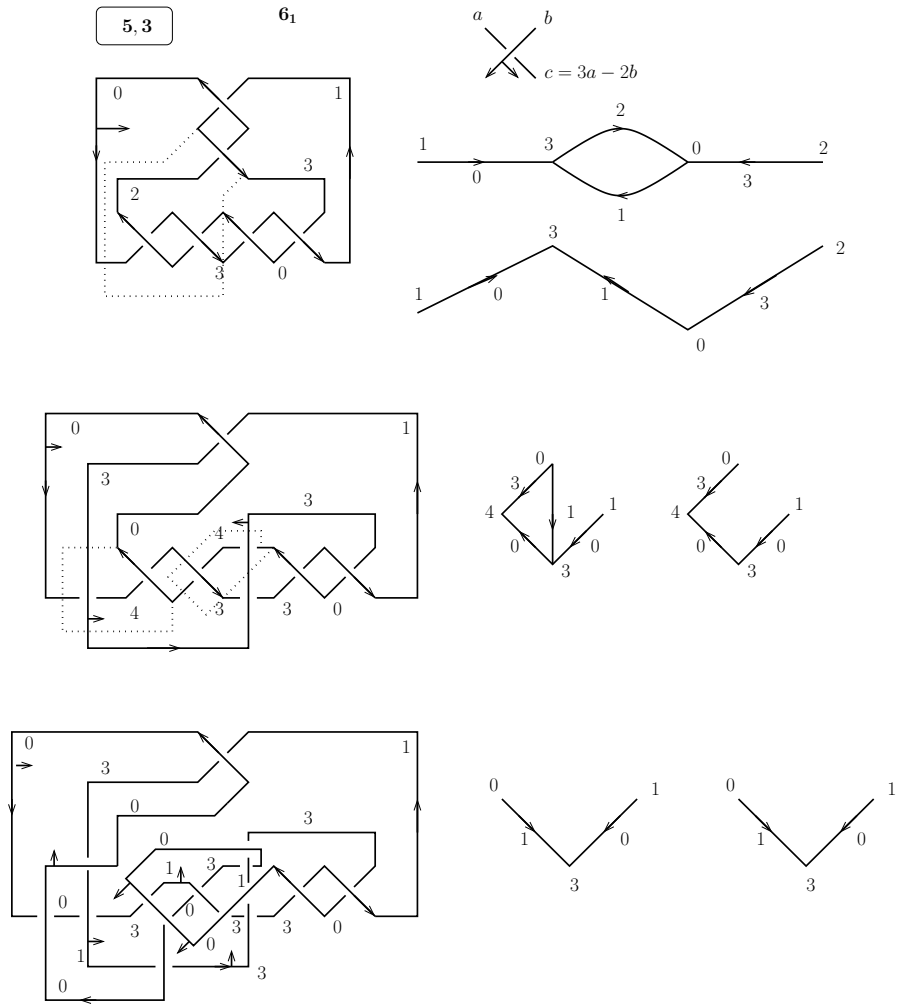}}}
	\caption{Revisiting. The knot $6_1$ whose $m$-determinant is $2m^2-5m+2$. For $m=3$ we obtain $3-\det = 5$. In the top row, a coloring with integers and the coloring condition. Below we perform colored Reidemeister moves (indicated by the broken lines) in order to remove the colors. In the middle row, we remove color $2$ but gain color $4$. In the bottom row we remove color $4$ obtaining the minimum of three colors - see Proposition \ref{prop:min3}.}\label{fig:6-1-m-3-reducing}
\end{figure}

\begin{figure}[!ht]
	\psfrag{0}{\huge$0$}
	\psfrag{1}{\huge$1$}
	\psfrag{2}{\huge$2$}
	\psfrag{3}{\huge$3$}
	\psfrag{4}{\huge$4$}
	\psfrag{5}{\huge$5$}
	\psfrag{6}{\huge$6$}
	\psfrag{a}{\huge$a$}
	\psfrag{b}{\huge$b$}
	\psfrag{2a-b}{\huge$c=2a-b$}
	\psfrag{6-3}{\huge$\mathbf{6_3}$}
	\psfrag{7, 2}{\huge$\mathbf{7, 2}$}
	\psfrag{m=2}{\huge$\mathbf{m=2}$}
	\centerline{\scalebox{.35}{\includegraphics{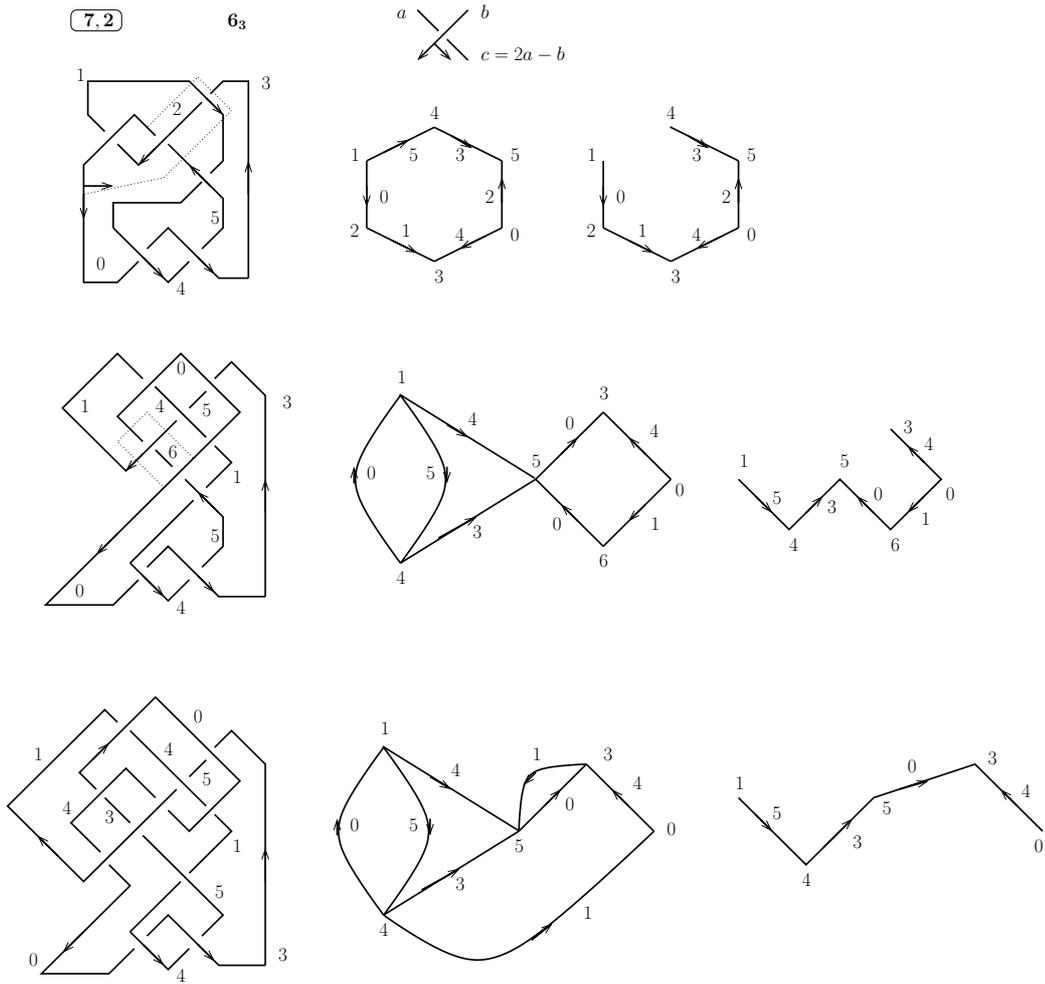}}}
	\caption{The knot $6_3$ whose $m$-determinant is $m^4-3m^3+5m^2-3m+1$. For $m=2$ we obtain $2-\det = 7$.   The initial diagram has 6 colors and the final diagram has 5 colors. Note the evolution of the corresponding palette graphs and spanning trees.}\label{fig:6-3-reducing}
\end{figure}

\section{Directions for future work.}\label{sect:futurework}

\begin{itemize}
\item Are palette graphs of non-trivial non-integral colorings connected? They are for knots thanks to Lemma \ref{lem:palettegraphknotsconnected}.

\item
Are there further conditions to be met so that the Kauffman-Harary result generalizes to the other linear Alexander quandles?

\item Is there a strong version of the KH behavior in the sense that if one reduced diagram of the alternating knot of prime $m$-determinant displays KH behavior then so do all other reduced alternating diagrams of the same knot?

\item Does $mincol_{p, m} (K)$ depend on $K$?

\item Is there a topological interpretation of these minima? Perhaps having to do with covering spaces of the knot?  Would such a topological interpretation give insight to the previous question?

\item What is more efficient at telling knots apart? The mere counting of colorings or the more complex state-sum invariant constructed using quandle cocycles?

\end{itemize}

\end{document}